\documentclass[11pt]{amsart}

\usepackage{amsmath}
\usepackage{amssymb}
\usepackage{amsthm}
\usepackage[all,cmtip]{xy}
\usepackage{color}

\theoremstyle{plain}
\newtheorem{thm}{Theorem}[section]
\newtheorem{prop}[thm]{Proposition}
\newtheorem{lem}[thm]{Lemma}
\newtheorem{cor}[thm]{Corollary}

\theoremstyle{definition}
\newtheorem{defn}[thm]{Definition}
\newtheorem{defnprop}[thm]{Definition-Proposition}
\newtheorem{eg}[thm]{Example}
\newtheorem{set}[thm]{Settings}

\theoremstyle{remark}
\newtheorem{rem}[thm]{Remark}
\newtheorem{claim}[thm]{Claim}

\newcommand{\bC}{\ensuremath{\mathbb{C}}}

\newcommand{\bF}{\ensuremath{\mathbb{F}}}

\newcommand{\bH}{\ensuremath{\mathbb{H}}}

\newcommand{\bP}{\ensuremath{\mathbb{P}}}
\newcommand{\bQ}{\ensuremath{\mathbb{Q}}}
\newcommand{\bR}{\ensuremath{\mathbb{R}}}

\newcommand{\bZ}{\ensuremath{\mathbb{Z}}}

\newcommand{\cF}{\ensuremath{\mathcal{F}}}

\newcommand{\cH}{\ensuremath{\mathcal{H}}}

\newcommand{\cL}{\ensuremath{\mathcal{L}}}

\newcommand{\cO}{\ensuremath{\mathcal{O}}}

\newcommand{\cS}{\ensuremath{\mathcal{S}}}
\newcommand{\cT}{\ensuremath{\mathcal{T}}}

\newcommand{\cX}{\ensuremath{\mathcal{X}}}
\newcommand{\cY}{\ensuremath{\mathcal{Y}}}

\newcommand{\fH}{\ensuremath{\mathfrak{H}}}

\DeclareMathOperator{\Pic}{Pic}

\DeclareMathOperator{\Aut}{Aut}

\DeclareMathOperator{\Ker}{Ker}
\DeclareMathOperator{\Coker}{Coker}
\DeclareMathOperator{\Image}{Im}

\DeclareMathOperator{\Gr}{Gr}

\DeclareMathOperator{\id}{id}

\DeclareMathOperator{\prim}{prim}

\newcommand{\dps}{\displaystyle}
\newcommand{\ol}{\overline}

\makeatletter
\@namedef{subjclassname@2020}{\textup{2020} Mathematics Subject Classification}
\makeatother


\begin{document}

\title
[The $\partial \bar{\partial}$-lemma and Hard Lefschetz property]
{On Hodge structures  
of 
compact complex manifolds with semistable degenerations} 

\author{Taro Sano}
\address{Department of Mathematics, Faculty of Science, Kobe University, Kobe, 657-8501, Japan}
\email{tarosano@math.kobe-u.ac.jp}

\maketitle

\begin{abstract} 
	Compact K\"{a}hler manifolds satisfy several nice Hodge-theoretic properties such as the Hodge symmetry, the Hard Lefschetz property and 
	the Hodge--Riemann bilinear relations, etc. 
	In this note, we investigate when such nice properties hold on compact complex manifolds with semistable degenerations. 
	
	For compact complex manifolds which can be obtained as smoothings of SNC varieties without triple intersection locus,  
	we show the Hodge symmetry when the monodromy logarithm induces isomorphisms on the associated graded pieces of the weight filtrations of the limiting mixed Hodge structures. 
	We also show the Hodge--Riemann relations on $H^3$ of compact complex 3-folds with such semistable degenerations under some conditions. 
		\end{abstract}

\tableofcontents

\section{Introduction} 
Smooth projective varieties (or more generally K\"{a}hler manifolds) satisfy several nice Hodge theoretic properties. 
For example, their de Rham cohomology admits pure Hodge structures and ample (K\"{a}hler) classes on them 
define positive definite bilinear form on their primitive cohomology groups. 

There are plenty of non-K\"{a}hler compact complex manifolds whose cohomology groups has nice properties as K\"{a}hler manifolds. 
Clemens and Friedman \cite{MR1141199} constructed non-K\"{a}hler Calabi--Yau 3-folds with $b_2 = 0$ and arbitrarily large $b_3$ (so called Clemens manifolds). 
Friedman \cite{MR4085665} and Li \cite{MR4735955}  showed that the $\partial\bar{\partial}$-lemma holds on Clemens manifolds. 
Furthermore, Li \cite{MR4735955} showed that an analogue of the Hodge--Riemann bilinear relation holds on them as well. 

On the other hand, Hashimoto and the author (\cite{MR4584262}, \cite{MR4406696}) constructed non-K\"{a}hler Calabi--Yau manifolds $X(a)$ in dimension $n \ge 3$ with arbitrarily large $b_2$ by smoothing algebraic SNC varieties. 
Note that Clemens manifolds have algebraic dimension $0$ since the Picard rank is $0$, while some of our examples $X(a)$ have algebraic dimension $n -2$ (cf. \cite[Theorem 1.1]{MR4406696}). 
They are expected to have nice cohomological properties since they are constructed from algebraic objects. 
The motivation of this note is to show that those examples also satisfy nice Hodge theoretic properties as Clemens manifolds and investigate what we really need for such statements. 

\begin{set}\label{set:intro}
Let $X= \bigcup_{i=1}^m X_i$  be a {\it proper SNC variety}, that is, a proper simple normal crossing $\bC$-scheme which is a union of smooth proper varieties $X_i$'s.  We consider its {\it semistable smoothing} $\phi \colon \cX \to \Delta$, 
that is, $\phi$ is a proper surjective holomorphic map from a complex manifold $\cX$ to a $1$-dimensional disk $\Delta$ such that $\phi^{-1}(0) \simeq X$. Let $n:= \dim X$. 

On a general smooth fiber $X_t:= \phi^{-1}(t)$ for $0 \neq t \in \Delta$, for $k >0$, 
it is known that the cohomology group $H^k:=H^k(X_t, \bC)$ admits a limiting mixed Hodge structure $(H^k, W_{\bullet}, F^{\bullet})$ constructed by Steenbrink \cite{MR0429885}. 
One of the important properties is that the monodromy logarithm $N_k:= \log T_k \colon H^k \to H^k$ 
is a morphism of mixed Hodge structures of type $(-1,-1)$. As a consequence, for $i >0$,  
it induces a homomorphism 
\[
N_k^{[i]} \colon \Gr^W_{k+i} H^k \to \Gr^W_{k-i} H^k, 
\] 
where $\Gr^W_l H^k:= W_l/W_{l-1}$ is the associated graded piece for the weight filtration $W_{\bullet} \subset H^k$. 
\end{set}

\begin{rem}
We often assume that $X$ is {\it without triple intersection}, that is, $X_i \cap X_j \cap X_k = \emptyset$ for all $i< j < k$ for simplicity, but it is still quite general setting. 
When $X$ is without triple intersections, the weight filtraions on $H^k$ are of the form $0=W_{k-2} \subset W_{k-1} \subset W_k \subset W_{k+1}= H^k$, thus only $N_k^{[1]}$ can be non-zero.  
\end{rem}

It is known that $N_k^{[i]}$ is an isomorphism for all $k$ and $i$ (see \cite[(5.9) Theorem]{MR0429885} and \cite[(5.2) Th\'{e}or\`{e}me]{MR1068383}) when $\phi$ is projective (or $X$ is cohomologically K\"{a}hler in the sense of \cite[(3.1)]{MR1068383}). In general, $N_k^{[i]}$ is not an isomorphism when $\phi$ is not projective. 
However, if it is an isomorphism, then we have the following. 

\begin{thm}\label{thm:main}
Let $X = \bigcup_{i=1}^m X_i$ be a proper SNC variety without triple intersection.  
Let $\phi \colon \cX \rightarrow \Delta$, $X_t$ and other notations be as in Settings \ref{set:intro}. 
Let $k>0$ and assume that $N_k^{[1]} \colon \Gr^W_{k+1}H^k \to \Gr^W_{k-1} H^k$ is an isomophism. 

Then $H^k = H^k(X_t, \bC)$ admits a pure Hodge structure (induced as in Proposition \ref{prop:pureHodgekopposed} and Remark \ref{rem:putativeHodge}(i)) for a small $t \in \Delta^* = \Delta \setminus \{0 \}$.  
\end{thm}

Although it is not easy to check that $N_k^{[i]}$ is an isomorphism in general, we have the following criterion on $H^n$ ($n=\dim X$) and its consequence. 

\begin{thm}\label{thm:intro2} 
Let $X = \bigcup_{i=1}^m X_i, \cX \to \Delta$ and others be as in Theorem \ref{thm:main}. 
Let $n:= \dim X$, $X^{(1)}:= \coprod_i X_i$ and $X^{(2)}:= \coprod_{i < j} X_i \cap X_j$ be the double locus of $X$. 
Assume that the cup product is non-degenerate on $\Image \rho_{n-1} \subset H^{n-1}(X^{(2)}, \bC)$, where $\rho_{n-1} 
\colon H^{n-1}(X^{(1)}, \bC) \to H^{n-1}(X^{(2)}, \bC)$ is the restriction map as in Definition \ref{defn:RestGysin}. 

Then $N_n^{[1]}$ is an isomorphism.  
As a consequence,  $H^n(X_t, \bC)$ admits a pure Hodge structure (induced as in Proposition \ref{prop:pureHodgekopposed} and Remark \ref{rem:putativeHodge}(i)) for a small $t \in \Delta^*$.  
\end{thm}

The examples constructed in \cite{MR4584262} are also covered as a consequence of the following.  

\begin{cor}\label{cor:Tyurincase}
Let $X=X_1 \cup X_2$ be a proper SNC 3-fold such that $X_1$ and $X_2$ are smooth projective and $D:= X_1 \cap X_2$ is irreducible. 
Assume that there is a semistable smoothing $\phi \colon \cX \to \Delta$ of $X$ and let $X_t := \phi^{-1}(t)$ be its general fiber over a small $t \in \Delta^*$.  Then the following holds. 

\begin{enumerate}
\item[(i)] The cohomology group $H^3(X_t, \bC)$ admits a pure Hodge structure.  
\item[(ii)] In particular, the non-K\"{a}hler Calabi--Yau 3-folds constructed in \cite{MR4584262} and \cite{Sano:2022aa} satisfy the $\partial\ol{\partial}$-lemma. 
\end{enumerate}
\end{cor}

While Theorem \ref{thm:main} is proved without the assumption on triple intersection in \cite{Chen:2024aa}, 
Theorem \ref{thm:intro2} and Corollary \ref{cor:Tyurincase}(i) are new parts of this note. 

\vspace{2mm}

Moreover, we obtain the Hodge--Riemann relation on $H^3$ of a smoothing of a proper SNC 3-fold with some conditions as follows. 

\begin{thm}\label{thm:Hodge-RiemannH^3}
Let $ \phi \colon \cY \to \Delta$ be a semistable smoothing of a proper SNC variety $Y= \bigcup_{j=1}^m Y_j$ without triple intersection such that 
$\dim Y =3$.  
Let $\rho_l \colon H^l(Y^{(1)}, \bQ) \to H^l(Y^{(2)}, \bQ)$ be the restriction map and $\gamma_l \colon H^{l-2}(Y^{(2)}, \bQ)(-1) \to H^l(Y^{(1)}, \bQ)$ be 
the Gysin map as in Definition \ref{defn:RestGysin}. Assume the following condition: 
\begin{itemize}
\item[(*)] The pairings $Q_{W_3}$ and $Q_{W_2}$ on $(\Ker \rho_3/ \Image \gamma_3)_{\bC}^{2,1}$ and $(\Ker \gamma_4)_{\bC}^{1,1}$ are positive definite  (see Remark \ref{rem:condition*explain} for the detail).
\end{itemize}
Let $Y_s = \phi^{-1}(s)$ be a fiber of $\phi$ over a small $s\neq 0$. 
For the cup product $\langle \bullet , \bullet \rangle$ on $H^3:=H^3(Y_s, \bC)$, 
let  $Q(\bullet, \bullet):= - \sqrt{-1}\langle \bullet , \bullet \rangle$. 

Then $H^3(Y_s)$ admits a pure Hodge structure and we have  $Q (\eta, \ol{\eta})  >0$ for $0 \neq \eta \in H^{2,1}_s:= H^{2,1}(Y_s, \bC)$. Hence $Q$ is a polarization on $H^3(Y_s)$. 
\end{thm}

As a consequence, the Hodge--Riemann bilinear relation holds on $H^3(Y_s)$ in Theorem \ref{thm:Hodge-RiemannH^3} since the relation is well-known on $H^{3,0}$.  
As a corollary of Theorem \ref{thm:Hodge-RiemannH^3}, we have the following. 

\begin{cor}\label{cor:Tyurinh^1=0case}
Let $Y = Y_1 \cup Y_2$ be a proper SNC variety such that $\dim Y=3$, $Y_j$ is projective with $H^1(Y_j, \cO_{Y_j}) =0$ for $j=1,2$, and $Y_1 \cap Y_2$ is irreducible. 
Let $\cY \to \Delta$ be a semistable smoothing of $Y$ and $Y_s$ be its general fiber over a small $s \in \Delta^*$.     

Then $H^3(Y_s, \bC)$ admits a pure Hodge structure and we have  $Q (\eta, \ol{\eta})  >0$ for $0 \neq \eta \in H^{2,1}_s:= H^{2,1}(Y_s, \bC)$. In particular, we have the Hodge--Riemann relation on 
$H^3$ of the non-K\"{a}hler Calabi--Yau 3-folds constructed in \cite{MR4584262}. 
\end{cor}

See \cite{Chen:2024aa} and \cite{Lee:2024aa} for more results in this direction. 
The strategy of the proof of Theorems \ref{thm:main} and \ref{thm:Hodge-RiemannH^3} follow those in 
\cite{MR4085665} and \cite{MR4735955} 
which use the limiting mixed Hodge structures on semistable smoothings in an effective way. 

\vspace{2mm}

In Section 2, we summarize the necessary materials on Hodge theory. 
In Section 3, we give the proof of Theorems \ref{thm:main} and \ref{thm:Hodge-RiemannH^3}. 
In Section 4, we exhibit some examples and investigate the Lefschetz line bundles and the Hodge--Riemann line bundles (cf. \cite{MR3224564}, \cite{MR4563867}) 
on non-K\"{a}hler Calabi--Yau manifolds. 
In particular, we show that the examples in \cite{MR4584262} satisfy the Hard Lefschetz property by exhibiting Lefschetz line bundles on them (Example \ref{eg:HashimotoSano}). 

\begin{rem}
When the author was finishing the manuscript, 
the preprints  \cite{Chen:2024aa} and \cite{Lee:2024aa} appeared on arXiv. 
The results are obtained independently. 
\end{rem}

\section{Preliminaries}

As is well-known, the existence of a pure Hodge structure is equivalent to the $k$-opposed property of the Hodge filtration (cf. \cite{MR0498551}). 

\begin{prop}\label{prop:pureHodgekopposed}
Let $k \in \bZ_{\ge 0}$ and let $H^k$ be a finite dimensional $\bC$-vector space. 
Let $F^{\bullet} \subset H^k$ be a decreasing filtration of linear subspaces such that 
\[
H^k=F^0 \supset F^1 \supset \cdots \supset F^k \supset F^{k+1}=0 
\]
 and let $H^{p,k-p}:= F^p \cap \ol{F^{k-p}}$ for $p \in \bZ$. 
Then the following are equivalent. 
\begin{enumerate}
\item[(i)]The subspaces $\{ H^{p,k-p} \mid 0 \le p \le k \}$ define a pure Hodge structure of weight $k$ on $H^k$, 
that is,   
\begin{equation}\label{eq:Hodgedecomp}
H^k = \bigoplus_{p=0}^k H^{p,k-p}, \ \ \ol{H^{p,k-p}} = H^{k-p,p}. 
\end{equation}
\item[(ii)] $H^k = F^p \oplus \ol{F^{k-p+1}}$ for all $p$ (that is, the filtration $F^{\bullet}$ is ``$k$-opposed'').  
\end{enumerate}
\end{prop}

\begin{proof} We write the proof for the convenience of the reader. 

\noindent\textbf{(i)$\Rightarrow$ (ii)} Suppose that the subspaces $\{H^{p, k-p} \}$ define a Hodge structure. 
Then, for all $p$, we obtain 
\[
\{0 \} = H^{p-1, k-p+1} \cap H^{p, k-p} = (F^{p-1} \cap \ol{F^{k-p+1}}) \cap (F^p \cap \ol{F^{k-p}}) =F^p \cap \ol{F^{k-p+1}}. 
\]
Moreover, we have 
\[
F^p \supset \bigoplus_{r \ge p} H^{r, k-r}
\]
by the definition of $H^{r,k-r}$. 
Similarly, we also have 
\[
\ol{F^{k-p+1}} \supset \ol{\bigoplus_{s \ge k-p+1} H^{s,k-s}} = \bigoplus_{s \ge k-p+1} H^{k-s,s} = \bigoplus_{r < p} H^{r,k-r}.  
\]
These show that $F^p + \ol{F^{k-p+1}} = \bigoplus_r H^{r,k-r}= H^k$. 

%
%

\vspace{2mm}

\noindent\textbf{(ii)$\Rightarrow$ (i)}
It is reduced to show the following: 

\begin{claim}\label{claim:decompF^p}
 In the setting (ii), we have 
$
F^p = F^{p+1} \oplus H^{p,k-p} 
$ for all $p$. 
\end{claim}

\begin{proof}[Proof of Claim]
Since we have $F^{p+1} \cap H^{p, k-p} = F^{p+1} \cap \ol{F^{k-p}}=0$, 
it is reduced to show $F^p = F^{p+1} + H^{p,k-p}$. 

Let $x \in F^p$. By $H^k = F^{p+1} \oplus \ol{F^{k-p}}$, 
there exist $y \in F^{p+1}$ and $z \in \ol{F^{k-p}}$ such that $x = y+z$. 
By 
$
z = x-y \in F^p
$, we see that $z \in H^{p, k-p}$ and $x \in F^{p+1} + H^{p,k-p}$. 
This shows $F^p \subset  F^{p+1} + H^{p,k-p}$ and the converse inclusion is clear.  
\end{proof}

By using Claim \ref{claim:decompF^p} repeatedly, we obtain (i) since we have 
\[
H^k = F^0 = F^1 \oplus H^{0,k} = \cdots =H^{k,0} \oplus \cdots \oplus H^{0,k}. 
\]
%
\end{proof}

\begin{rem}\label{rem:putativeHodge}
\begin{enumerate}
\item[(i)] Let $k \in \bZ_{\ge 0}$ and let $X$ be a compact complex manifold. 
For $p \in \bZ$, let 
\[
F^p:=F^p H^k(X, \bC) := \Image (\bH^k(\sigma^{\ge p} \Omega_X^{\bullet}) \to H^k(X, \bC)) 
\]
be the putative Hodge filtration as in \cite[Definition 2.21]{MR2393625}. 
Then $H^k=H^k(X, \bC)$ with $F^{\bullet}$ satisfy the hypothesis of Proposition \ref{prop:pureHodgekopposed}. 

The $\partial \ol{\partial}$-lemma holds for $X$ if and only if the Hodge-de Rham spectral sequence 
\[
E_1^{p,q}=H^q(X, \Omega_X^p)  \Rightarrow H^{p+q}(X, \bC)
\]
degenerates at $E_1$ and, for all $k \in \bZ_{\ge 0}$, the filtration $F^{\bullet} \subset H^k$ is $k$-opposed as in Proposition \ref{prop:pureHodgekopposed}(ii) (cf. \cite[(5.21)]{MR382702}, \cite[Theorem 1.2]{MR4085665}). 
\item[(ii)] The condition $F^p \cap \ol{F^{k-p+1}}=0$ for all $p$ is not equivalent to (i) and (ii) in Proposition \ref{prop:pureHodgekopposed}. For example, let $X$ be a Hopf surface. 
Then $H^1= H^1(X, \bC)$ has the Hodge filtration $\bC \simeq H^1=F^0 \supset F^1 =0$. 
Then $F^1 + \ol{F^1}=0= F^1 \cap \ol{F^1}$, thus we do not have the Hodge decomposition as in Proposition \ref{prop:pureHodgekopposed}. (Note that a sentence in \cite[pp.34, the bottom]{MR2393625} is a bit misleading.)  
\end{enumerate}
\end{rem}

We have the following description of limiting mixed Hodge structures on the central fiber of a semistable degeneration. 

\begin{thm}\label{thm:hodgejunbi}
Let $\phi \colon \cX \rightarrow \Delta$ be a semistable smoothing of a proper SNC variety $X$ over a $1$-dimensional unit disk. 
Let $\Lambda^{\bullet}_{X}:= \Omega^{\bullet}_{\cX/ \Delta}(\log X)|_{X}$. Let $\Delta^*:= \Delta \setminus \{0 \}$, $\cX^*:= \phi^{-1}(\Delta^*)$ and $\phi' \colon \cX^* \to \Delta^*$ be the smooth family induced by $\phi$. 
\begin{enumerate}
\item[(i)] The hypercohomology $H^k:=H^k_{\lim}:= \bH^k( X, \Lambda^{\bullet}_{X})$ is isomorphic to 
$H^k(X_t, \bC)$, where $X_t:= \phi^{-1}(t)$($t \neq 0$) is the general fiber of $\phi$. The sheaf 
$$\cH^k:= \bR^k \phi_* \Omega^{\bullet}_{\cX/ \Delta} (\log X)$$ is locally free and coincides with Deligne's canonical extension of $\cH^k|_{\Delta^*} \simeq R^k \phi'_* \bC_{\cX^*} \otimes \cO_{\Delta^*}$ for the Gauss--Manin connection. 
\item[(ii)] There is a $\bQ$-mixed Hodge structure $(H^k, H^k_{\bQ}, W_{\bullet}^{\lim}, F^{\bullet}_{\lim})$ on $H^k$. 
The spectral sequence 
\[
{}_F E_1^{p,q} = H^q(X, \Lambda_{X}^{p}) \Rightarrow H^{p+q}=\bH^{p+q}(\Lambda_{X}^{\bullet}) 
\]
degenerates at $E_1$ and induces the Hodge filtration $F^{\bullet}_{\lim}$. 
\item[(iii)] 
There exists the monodromy weight spectral sequence 
\[
{}_{W(M)} E_1^{-r, k+r} = \bigoplus_{l \ge \max \{0, -r \}} H^{k-r-2l} \left( X^{(r+2l+1)}, \bQ \right)(-r-l) \Rightarrow H^k_{\bQ}  
\]
which degenerates at $E_2$ and induces the weight filtration on $H^k_{\bQ}$, where $X^{(m)}$ is the $m$-fold intersection locus of $X$ for $m \ge 1$. 
(For $i \in \bZ$, $(i)$ means the Tate twist of degree $i$. )
\item[(iv)] Possibly after shrinking $\Delta$, the spectral sequence 
\[
E_1^{p,q}= R^q \phi_* \Omega^p_{\cX/ \Delta}( \log X) \Rightarrow \bR^{p+q} \phi_* \Omega^{\bullet}_{\cX/ \Delta}(\log X)
\] 
degenerates at $E_1$ and the sheaf $R^q \phi_* \Omega^p_{\cX/ \Delta}( \log X)$ is locally free. 
The spectral sequence 
\[
H^q(X_t, \Omega^p_{X_t}) \Rightarrow H^{p+q}(X_t, \bC)
\]
for $X_t= \phi^{-1}(t)$ degenerates at $E_1$. 
Moreover, the Hodge filtration on $H^k$ induces a filtration  $\cF^{\bullet} \subset \cH^k$ by holomorphic subbundles. 
\end{enumerate}
\end{thm}

\begin{proof}
\noindent(i) This follows from the arguments in \cite[Corollary 11.18]{MR2393625} (cf. \cite[Theorem 2.6]{MR4085665}, \cite[Theorem 2.1(i)]{MR4735955}). 

\noindent(ii) This follows from the arguments in \cite[Theorem 11.22, Corollary 11.23(ii)]{MR2393625}. 

\noindent(iii) follows from 
\cite[Corollary 11.23(i)]{MR2393625}

\noindent(iv) The $E_1$-degeneration on $X_t$ is in \cite[Corollary 11.24]{MR2393625}. Then the $E_1$-degeneration for the sheaves follows from the argument to show 
 \cite[Proposition 10.29]{MR2393625}.     
\end{proof}

\begin{rem}
Let $\phi \colon \cX \to \Delta$ be a semistable smoothing as in Theorem \ref{thm:hodgejunbi} such that $X:= \phi^{-1}(0)$ is an SNC divisor 
with $\omega_{X} \simeq \cO_{X}$. Then $\phi_* \omega_{\cX/ \Delta}$ is locally free 
and the dimension $h^0(\cX_t, \omega_{\cX_t})$ is constant on $\cX_t = \phi^{-1}(t)$ for $t \in \Delta$. 
Hence we see that $\omega_{\cX_t} \simeq \cO_{\cX_t}$.  
\end{rem}

\begin{rem}\label{rem:monodromylog}
Let $\cX \to \Delta$ be a semistable smoothing as in Theorem \ref{thm:hodgejunbi}. 
Then we have a monodromy transformation $T$ of $H^k = H^k(X, \Lambda^{\bullet}_{X})$ 
which is a natural extension of the monodromy transformation on smooth fibers. 
Then $T$ is unipotent (cf. \cite[Corollary 11.19]{MR2393625}) and we let  
\[
N:= \log T = -\sum_{l=1}^{\dim X} \frac{(\id-T)^l}{l} 
\]  
be the monodromy logarithm. 
Then it is known that \[
N':= \frac{N}{2 \pi \sqrt{-1}} \colon H^k_{\bQ} \to H^k_{\bQ}(-1)
\] is a morphism of mixed Hodge structures (cf. \cite[Theorem 11.28]{MR2393625}). 
\end{rem}

\begin{defn}\label{defn:RestGysin} (cf. \cite[(1.2), (1.3)]{MR1068383})
Let $X = \bigcup_{i=1}^m X_i$ be a proper SNC variety with irreducible components $X_1, \ldots , X_m$. 
Let $X^{(1)}:= \coprod_{i=1}^m X_i$ and $X^{(2)}:= \coprod_{i<j} X_{ij}$, where $X_{ij}:= X_i \cap X_j$. 
Let $\delta_1, \delta_2 \colon X^{(2)} \to X^{(1)}$ be the morphism induced by $X_{ij} \hookrightarrow X_j, X_{ij} \hookrightarrow X_i$ respectively. 

 We define the restriction map $\rho_l$ by  
\[
\rho_l:= \delta_1^* - \delta_2^* \colon H^l(X^{(1)}, \bQ) \to H^l (X^{(2)}, \bQ).  
\]
We also define $\gamma_l:= (\delta_1)_! - (\delta_2)_! \colon H^{l-2}(X^{(2)}, \bQ)(-1) \to H^{l}(X^{(1)}, \bQ)$ to be the Gysin map for $\rho_l$ as in \cite[(1.3)]{MR1068383} with the formula
\begin{equation}\label{eq:adjointRhoGamma}
\left( \frac{1}{2\pi \sqrt{-1}} \right)^{n-1} \int_{X^{(2)}} \alpha \cup \rho_{2n-l}(\beta) = \left(\frac{1}{2\pi \sqrt{-1}} \right)^{n} \int_{X^{(1)}} \gamma_l(\alpha) \cup \beta
\end{equation} 
for $n:= \dim X$, $\alpha \in H^{l-2}(X^{(2)})(-1)$ and $\beta \in H^{2n-l}(X^{(1)})$. 

Note that $\rho_l$ and $\gamma_l$ are morphism of Hodge structures, thus their kernels, images and cokernels admit Hodge structures. 
\end{defn}

When the central fiber is an SNC variety without triple intersection, we have the following description. 

\begin{prop}\label{prop:H3CYdesc}
Let $X$ be a proper SNC variety without triple intersection and $\phi \colon \cX \rightarrow \Delta$ be a semistable smoothing of $X$. 
Let $X^{(1)} = \coprod_{i=1}^m X_i$ and $X^{(2)}= \coprod_{i < j} (X_i \cap X_j)$ be as in Definition \ref{defn:RestGysin}. 
Let $H^k:= \bH^k(X, \Lambda^{\bullet}_{X})$ be as in Theorem \ref{thm:hodgejunbi}(i) and $H^k_{\bQ} \subset H^k$ be the $\bQ$-structure with the weight filtration $W_{\bullet}$. Then we have the following.  
\begin{enumerate}
\item[(i)] The term $E_1^{-r, k+r}$ in the weight spectral sequence is non-zero only if $r=-1,0,1$ and we have 
\[
E_1^{1, k-1} = H^{k-1}(X^{(2)}, \bQ), \ \ E_1^{0,k} = H^k(X^{(1)}, \bQ), \ \ E_1^{-1, k+1} = H^{k-1}(X^{(2)}, \bQ)(-1).   
\]
\item[(ii)] As a consequence of (i), we have $W_{k-2} =0$, $W_{k+1} = H^k_{\bQ}$ and 
\[
W_{k-1} \simeq \Coker (E_1^{0,k-1} \to E_1^{1, k-1}) = \Coker \rho_{k-1}, 
\]
\[
W_k/W_{k-1} \simeq \Ker \rho_k / \Image \gamma_k, 
\]
\[
W_{k+1}/ W_k \simeq \Ker \gamma_{k+1}, 
\]
where $\rho_l \colon H^l(X^{(1)}, \bQ) \to H^l(X^{(2)}, \bQ)$ is the restriction map and 
$\gamma_l \colon H^{l-2}(X^{(2)}, \bQ)(-1) \to H^l(X^{(1)}, \bQ)$ is the Gysin map as in Definition \ref{defn:RestGysin}.  

\item[(iii)] Assume that $k=n:=\dim X$.   
 Then the monodromy nilpotent operator $N' \colon H^n_{\bQ} \to H^n_{\bQ}(-1)$ as in Remark \ref{rem:monodromylog} induces an isomorphism 
$$
N' \colon W_{n+1}/W_n \xrightarrow{\simeq} W_{n-1}(-1) 
$$ 
if and only if the cup product is non-degenerate on $\Image \rho_{n-1} \subset H^{n-1}(X^{(2)}, \bC)$. 
\end{enumerate} 
\end{prop}

\begin{proof}
(i) follows from Theorem \ref{thm:hodgejunbi}(iii) since we only need to care $r+2l+1 =1,2$. 

\noindent(ii) is a consequence of (i) and the weight spectral sequence. 

\noindent(iii) By (ii), we see that the homomorphism $N' \colon W_{n+1}/W_n \rightarrow W_{n-1}(-1)$ can be regarded as a homomorphism 
\[
N' \colon \Ker \gamma_{n+1} \to \Coker \rho_{n-1} (-1)
\] 
induced by the identity homomorphism on $H^{n-1}(X^{(2)}, \bQ)(-1)$. 
Let 
\[
(\Image \rho_{n-1})^{\perp} \subset H^{n-1}(X^{(2)}, \bQ)
\]
 be the orthogonal complement of $\Image \rho_{n-1}$ for the cup product.  
 Then we see that $$\Ker \gamma_{n+1} = (\Image \rho_{n-1})^{\perp}(-1)$$ by the adjoint property 
 of $\rho_{n-1}$ and $\gamma_{n-1}$ as in (\ref{eq:adjointRhoGamma}) and the non-degeneracy of the cup product on $H^{n-1}(X^{(2)})$. 
 
Assume that the cup product is non-degenerate on $\Image \rho_{n-1}$.  
Then the natural homomorphism 
 \[
\pi' \colon (\Image \rho_{n-1})^{\perp} \to \Coker \rho_{n-1}
 \]
 is an isomorphism. 
 (If $(V, \langle , \rangle)$ is a finite dimensional vector space with a non-degenerate bilinear form 
 and $W \subset V$ is a subspace such that $\langle , \rangle|_W$ is non-degenerate, then the composition $W^\perp \hookrightarrow V \to V/W$  is an isomorphism.) 
 Hence $N'$ is an isomorphism. 
 
 On the other hand, if the cup product is degenerate on $\Image \rho_{n-1}$, then the homomorphism $\pi'$ is not injective. 
 Hence we obtain the converse. 
\end{proof}

\begin{rem}\label{rem:LefschetzCompati}
Let $X$ be a smooth projective variety of pure dimension $n$ with an ample class $L \in H^2(X, \bZ)$. 
Then, for $\alpha, \beta \in H^n(X, \bC)$, we have the non-degenerate pairing
\[
\langle \alpha, \beta \rangle:= \int_X \alpha \wedge \beta.  
\]
Let $W \subset H^n(X, \bC)$ be a sub-Hodge structure which is ``compatible with the Lefschetz decomposition'', that is, 
\[
W= \bigoplus_{r\ge0} \left( L^r H^{n-2r}_{\prim}(X) \cap W \right),     
\] 
where $H^{n-2r}_{\prim}(X):= \Ker (L^{2r+1} \colon H^{n-2r}(X, \bC) \to H^{n+2r+2}(X, \bC))$ is the primitive cohomology for $L$. 
Then we see that $\langle , \rangle|_W$ is also non-degenerate. 

\begin{proof} 

Note that we have the Hodge decomposition 
\[
L^r H^{n-2r}_{\prim}(X) \cap W = \bigoplus_{p=r}^{n-r} \left( L^r H^{n-2r}_{\prim}(X) \cap W^{p,n-p} \right)
\] 
and the decomposition is orthogonal for the pairing. 
We have the polarization on each summand $L^r H^{n-2r}_{\prim}(X) \cap W^{p,n-p}$ defined by 
\[
Q(L^r(\gamma), L^r(\delta)) = (-1)^{\frac{(n-2r)(n-2r-1)}{2}} \sqrt{-1}^{2p-n} \int_X  \gamma \wedge \ol{\delta} \wedge L^{2r}. 
\]
This implies the non-degeneracy of the pairing on $W$. 
\end{proof}

By the above, if $\Image \rho_{n-1}$ is compatible with the Lefschetz decomposition, 
then the cup product is non-degenerate on $\Image \rho_{n-1}$. 
\end{rem}

We shall use the Deligne's splitting for the mixed Hodge structure. 

\begin{defnprop}\label{defnprop:MHS}
Let $(H, W_{\bullet}, F^{\bullet})$ be a mixed Hodge structure. Let 
\[
I^{p,q}:= F^p \cap W_{p+q} \cap (\ol{F^q}\cap W_{p+q} + \sum_{j \ge 2} \ol{F^{q-j+1}} \cap W_{p+q-j} ). 
\]
Then we have 
\[
W_k = \bigoplus_{p+q \le k} I^{p,q}, \ \ F^p = \bigoplus_{r \ge p} I^{r,q}. 
\]
\end{defnprop}

For a smoothing as in Proposition \ref{prop:H3CYdesc}, we have the following information on $I^{p,q}$. 
	
\begin{prop}\label{prop:limitMHSIpq}
Let $X$ be a proper SNC variety and $\phi \colon \cX \rightarrow \Delta$ as in Proposition \ref{prop:H3CYdesc}. 
Let $(H^k, W_{\bullet}, F^{\bullet})$ be the limiting mixed Hodge structure on $H^k$ as in Theorem \ref{thm:hodgejunbi}(ii). 
Then  the following holds. 
\begin{enumerate}
\item[(i)] We have 
\[
W_{k-1} = I^{k-1,0} \oplus I^{k-2,1} \oplus \cdots \oplus I^{0,k-1}, \ \ 
\ol{I^{i,k-1-i}} = I^{k-1-i, i} \ (i=0, \ldots k-1), 
\]
\[
W_k=W_{k-1} \oplus I^{k,0} \oplus I^{k-1, 1} \oplus \cdots \oplus I^{0,k}, \ \ \ol{I^{i,k-i}} = I^{k-i,i} \ (i=0, \ldots ,k), 
\]
\[
W_{k+1}= W_k \oplus I^{k,1}  \oplus \cdots \oplus I^{1, k},  
\]
\item[(ii)] 
Moreover, for all $p$, we have 
\begin{equation}\label{eq:F^pI^pqdecomp}
F^p = \left( I^{p, k-1-p} \oplus \cdots \oplus I^{k-1, 0} \right) \oplus \left( I^{p, k-p} \oplus \cdots  \oplus I^{k, 0} \right) \oplus 
\left( I^{p, k+1-p} \oplus \cdots \oplus I^{k, 1} \right). 
\end{equation}
\item[(iii)] Let $\pi \colon W_{k+1} \to W_{k+1}/W_k$ be the projection and $V^{i, k+1-i}:= \pi (I^{i, k+1-i})$ for $i=1, \ldots , k$. 
Then we have the decomposition $$W_{k+1}/W_k = \bigoplus_{i=1}^k V^{i, k+1-i}$$ 
such that $\ol{V^{i, k+1-i}} = V^{k+1-i,i}$. 
\end{enumerate}
\end{prop}	

\begin{proof}
\noindent(i),(ii): The decompositions of $W_{k-1}$, $W_k$ and $F^p$ follow from Definition-Proposition \ref{defnprop:MHS} and Proposition \ref{prop:H3CYdesc}(ii).  
We see that 
\[
I^{i, k-1-i} = F^i \cap \ol{F^{k-1-i}} \cap W_{k-1}, \ \  
I^{i, k-i} = F^i \cap \ol{F^{k-i}} \cap W_{k}
\]
by $W_{k-1-j} =0$ and $W_{k-j} = 0$ for $j \ge 2$. 
 Hence we obtain $\ol{I^{i,k-1-i}} = I^{k-1-i, i}$ and $\ol{I^{i,k-i}} = I^{k-i,i}$. 
Since $W_{k+1}/W_k \simeq \Ker \gamma_k$ is a sub-Hodge structure of $H^{k-1}(X^{(2)}, \bQ)(-1)$ by Proposition \ref{prop:H3CYdesc} (ii), 
we see that $I^{k+1, 0} = I^{0, k+1} =0$. 

\noindent(iii) follows since $W_{\bullet}$ induces a mixed Hodge structure. 
\end{proof}

\section{Proof of Theorems}

\subsection{$\partial\bar{\partial}$-lemmas}

By using the above ingredients, we prove Theorem \ref{thm:main}. 

\begin{proof}[Proof of Theorem \ref{thm:main}] 
By 
the hypothesis, 
the monodromy logarithm induces an isomorphism 
\[
N = N_k^{[1]} \colon W_{k+1}/W_k \simeq W_{k-1}. 
\]

By \cite[Theorem 1.2]{MR4085665} and Theorem \ref{thm:hodgejunbi}(iv), it is enough to show that the Hodge filtration $F_t^{\bullet}$ on $H^k(X_t, \bC)$ satisfies 
\begin{equation}\label{eq:k-opposed}
F^p_t \oplus \ol{F^{k-p+1}_t} = H^k(X_t, \bC) 
\end{equation}
for all $p$ on the smooth fiber $X_t = \phi^{-1}(t)$ of the semistable smoothing $\phi \colon \cX \to \Delta$. 

First, let us choose a $\bC$-basis of $F^p_0=F^p_{\lim} \subset H^k$ of the limit Hodge filtration. 
Let $h^{i,j}:= \dim I^{i,j}$. 
We may take a basis $\{ u^{i,j}_l \in I^{i,j} \mid l=1, \ldots ,h^{i,j}\}$ of $I^{i,j}$ for $i+j=k-1,k, k+1$ so that 
\[
\ol{u^{i, k-1-i}_l} = u^{k-1-i, i}_l \ \ (i=0,\ldots, k-1), 
\]  
\[
\ol{u^{i, k-i}_l} = u^{k-i, i}_l \ \ (i=0,\ldots, k), 
\]  
\[
N(u^{i,k+1-i}_l) = u^{i-1, k-i}_l \ \ (i=1,\ldots, k, l= 1,\ldots, h^{i,k+1-i}) 
\]
by Proposition \ref{prop:limitMHSIpq}(i) and the induced isomorphism $N \colon I^{i, k+1-i} \xrightarrow{\sim} I^{i-1, k-i}$. 

We may also assume that 
\begin{equation}\label{eq:olu^{p,k+1-p}w_l}
\ol{u^{i, k+1-i}_l} = u^{k+1-i, i}_l + w_l
\end{equation} for some $w_l = w_{l,i} \in W_k$ 
since $\pi(\ol{I^{i, k+1-i}}) = \pi( I^{k+1-i, i})$ for $\pi \colon W_{k+1} \to W_{k+1}/W_k$ as in Proposition \ref{prop:limitMHSIpq}(iii). 


\begin{rem}(cf. \cite{MR2393625}, \cite{Schnell:2007aa})
For any basis $v_1, \ldots ,v_r \in H^k = \bH^k(X, \Lambda_{X}^{\bullet})$, we have a local holomorphic frame $v_1(\zeta), \ldots , v_r(\zeta) \in \cH^k = \bR^k \phi_* \Omega^{\bullet}_{\cX/ \Delta} (\log X)$ 
with the variable $\zeta \in \Delta$ since $\cH^k$ is a trivial bundle. 
Let $\tau \colon \mathfrak{h} \to \Delta^*:= \Delta\setminus \{0 \}$ be the universal cover defined by $\tau(z):= e^{2\pi \sqrt{-1}z}$. 
These $v_i(\zeta)$   
correspond to the $H^k$-valued functions 
 $v_i(z) \colon \mathfrak{h} \rightarrow H^k$ for $i=1, \ldots , r$ defined by $v_i(z):= e^{zN} \cdot v_i$ with the invariance property 
 $v_i(z+1) = T v_i(z)$ for the monodromy transformation $T$ of $H^k$.  
The functions $v_1(z), \ldots , v_r(z)$ form a local holomorphic frame of $\tau^* \cH^k|_{\Delta^*}$ with the invariance property. 
 
 Since $\cF^p \subset \cH^k$ is a holomorphic subbundle, we can extend $u \in F^p \subset H^k$ to a holomorphic section $U(\zeta) \in \cF^p$ over $\Delta$  which corresponds to the function $U(z) \colon \fH \to H^k$ with the property $U(z+1) = T \cdot U(z)$. 
 Note that $U(\zeta)$ can be written as a $\cO_{\Delta}$-linear combination of $v_1(\zeta), \ldots , v_r(\zeta)$. 
\end{rem}

Note that the decomposition (\ref{eq:F^pI^pqdecomp}) implies 
\begin{multline*}
\ol{F^{k+1-p}} \\ 
= \ol{\left( I^{k+1-p, p-2} \oplus \cdots \oplus I^{k-1, 0} \right)} \oplus \ol{\left( I^{k+1-p, p-1} \oplus \cdots  \oplus I^{k, 0} \right)} \oplus 
\ol{\left( I^{k+1-p, p} \oplus \cdots \oplus I^{k, 1} \right)} \\
= \left( I^{ p-2, k+1-p} \oplus \cdots \oplus I^{ 0, k-1} \right) \oplus \left( I^{ p-1, k+1-p} \oplus \cdots  \oplus I^{ 0, k} \right) \oplus 
\left( \ol{I^{k+1-p, p}} \oplus \cdots \oplus \ol{I^{k, 1}} \right).
\end{multline*}

Let $U^{i,j}_l(z) \colon \fH \to H^k$ be the function corresponding to the section of $\cF^p$ which is an extension 
of $u^{i,j}_l \in F^p$ and let $u^{i,j}_l(z):= e^{zN} \cdot u^{i,j}_l \colon \fH \to H^k$ be the function determined by $u^{i,j}_l \in H^k$. Note that the corresponding sections $\{u^{i,j}_l(\zeta) \in \cH^k \}$ forms a local holomorphic frame of $\cH^k$ and we may write 
\begin{equation}\label{eq:UandAdescribe}
U^{i,j}_l(z) = \sum A^{i',j'}_{l'}(\zeta) \cdot u^{i',j'}_{l'}(z)
\end{equation}
for some holomorphic functions $A^{i',j'}_{l'}(\zeta) =A^{i',j'}_{l'}[l,i,j] \in \cO_{\Delta}$ (which depend on $l,i,j$). 
Since we have 
\[
u^{i,j}_l(z) = \begin{cases}
u^{i,j}_l & (i+j \le k) \\
u^{i,j}_l + z u^{i-1, j-1}_l & (i+j = k+1), 
\end{cases}
\]
we have 
\[
U^{i,j}_l(z) = \sum (A^{i',k-1-i'}_{l'}+ z A^{i'+1, k-i'}_{l'})\cdot u^{i',k-1-i'}_{l'} 
+ \sum A^{i',k-i'}_{l'} \cdot u^{i',k-i'}_{l'} + 
\sum A_{l'}^{i',k+1-i'} \cdot u^{i',k+1-i'}_{l'}. 
\]
Let $z = x+\sqrt{-1}y$. 
By the above description,  we see that 
\begin{equation}\label{eq:U^{i,k-1-i}approx}
U^{i, k-1-i}_l(z) = u_l^{i,k-1-i}+ O\left(y^{-1} \right).  
\end{equation}
Indeed, since $A^{i,j}_l$ is holomorphic, we have $A^{i,k-1-i}_l(\zeta) = 1+ O\left(y^{-1} \right)$, $A^{i',j'}_{l'}(\zeta) = O\left(y^{-1} \right)$ for $(i', j', l') \neq (i,k-1-i,l)$ and 
$\dps{|z \zeta| = \left| \frac{y}{e^y} \right|}$. Similarly, we have 
\[
U_l^{i,k-i}(z) = u^{i,k-i}_l+ O\left(y^{-1} \right), \ \ \ \ 
U_l^{i, k+1-i}(z) = u_l^{i, k+1-i}+z u^{i-1, k-i}_l + O\left(y^{-1} \right). 
\]
We also see that 
\[
\ol{U_l^{i, k-1-i}(z)} =u_l^{k-1-i,i}+ O\left(y^{-1} \right), \ \  
\ol{U_l^{i, k-i}(z)} = u_l^{k-i,i}+ O\left(y^{-1} \right), 
\]
\[
\ol{U_l^{i,k+1-i}(z)} = \ol{u_l^{i, k+1-i}} + \ol{z} \cdot u_l^{k-i, i-1} + O\left(y^{-1} \right). 
\]

In order to show (\ref{eq:k-opposed}), it is enough to check 
\begin{multline*}
\Phi(z):= \left( \bigwedge_{i=p}^{k-1} \bigwedge_{l=1}^{h^{i,k-1-i}} U_l^{i,k-1-i}(z) \right) \wedge \left( \bigwedge_{i=p}^{k} \bigwedge_{l=1}^{h^{i,k-i}} U_l^{i,k-i}(z) \right) \wedge 
 \left( \bigwedge_{i=p}^{k} \bigwedge_{l=1}^{h^{i,k+1-i}} U_l^{i,k+1-i}(z) \right) \\
 \wedge  \left( \bigwedge_{i=k+1-p}^{k-1} \bigwedge_{l=1}^{h^{i,k-1-i}} \ol{U_l^{i,k-1-i}(z)} \right) 
 \wedge  \left( \bigwedge_{i=k+1-p}^{k} \bigwedge_{l=1}^{h^{i,k-i}} \ol{U_l^{i,k-i}(z)} \right) 
  \wedge  \left( \bigwedge_{i=k+1-p}^{k} \bigwedge_{l=1}^{h^{i,k+1-i}} \ol{U_l^{i,k+1-i}(z)} \right) 
  \\ \neq 0  \in \bigwedge^{b_k} H^k
\end{multline*}
for $y=\Image z \gg 0$, where $b_k:= \dim_{\bC} H^k$.

By the above approximations as (\ref{eq:U^{i,k-1-i}approx}), we compute that 

\begin{multline}\label{eq:phi_z}
\Phi(z)  = 
 \bigwedge_{i=p}^{k-1}  \bigwedge_l \left( u_l^{i,k-1-i}+ O\left(y^{-1} \right) \right) \wedge \bigwedge_{i=p}^k \bigwedge_l \left( u_l^{i,k-i} + O\left(y^{-1} \right)\right) 
 \\ \wedge \bigwedge_{i=p}^k \bigwedge_l \left( u^{i,k+1-i}_l+ z \cdot u^{i-1, k-i}_l + O\left(y^{-1} \right) \right) \\ 
\wedge \bigwedge_{i=k+1-p}^{k-1} \bigwedge_l \left( u_l^{k-1-i,i} + O\left(y^{-1} \right) \right) \wedge 
\bigwedge_{i=k+1-p}^k \bigwedge_l \left( u_l^{k-i,i} + O\left(y^{-1} \right) \right) \\ 
\wedge \bigwedge_{i=k+1-p}^k \bigwedge_l 
\left( \ol{u_l^{i,k+1-i}} + \ol{z} \cdot u_l^{k-i, i-1} + O\left(y^{-1} \right) \right) \\
= \pm \left(\bigwedge_{i=0, i \neq p-1}^{k-1}  \bigwedge_l  u_l^{i,k-1-i} \right) \wedge \left( \bigwedge_{i=0}^k \bigwedge_l u^{i,k-i}_l \right) \wedge \bigwedge_{i=p+1}^k \bigwedge_l u_l^{i, k+1-i} \wedge \bigwedge_{i=k+2-p}^k \bigwedge_l \ol{u^{i, k+1-i}_l} 
\\ 
\wedge \bigwedge_l \left( u_l^{p, k+1-p} + z \cdot u_l^{p-1,k-p} \right) \wedge \left( \ol{u_l^{k+1-p,p}} + \ol{z} \cdot u^{p-1, k-p}_l \right) + O\left(y^{-1} \right) \\ 
= \pm \left(\bigwedge_{i=0, i \neq p-1}^{k-1}  \bigwedge_l  u_l^{i,k-1-i} \right) \wedge \left( \bigwedge_{i=0}^k \bigwedge_l u^{i,k-i}_l \right) \wedge \bigwedge_{i=p+1}^k \bigwedge_l u_l^{i, k+1-i} \wedge \bigwedge_{i=k+2-p}^k \bigwedge_l (u_l^{k+1-i,i} + w_{l,i})  
\\ 
\wedge \bigwedge_l \left( u_l^{p, k+1-p} + z \cdot u_l^{p-1,k-p} \right) \wedge \left( \ol{u_l^{k+1-p,p}} + \ol{z} \cdot u^{p-1, k-p}_l \right) + O\left(y^{-1} \right). 
\end{multline}
Note that we can ignore the term $w_{l,i}$ in the last line since the wedge products from these do not contain elements from $I^{i, k+1-i}$ for $i=1, \ldots, p-1$ and vanish.   

By (\ref{eq:olu^{p,k+1-p}w_l}), we have 
\begin{multline*}
\left( u_l^{p, k+1-p} + z \cdot u_l^{p-1,k-p} \right) \wedge \left( \ol{u_l^{k+1-p,p}} + \ol{z} \cdot u^{p-1, k-p}_l \right) \\ 
= \left( u_l^{p, k+1-p} + z \cdot u_l^{p-1,k-p} \right) \wedge \left( u_l^{p, k+1-p} +w_l + \ol{z} \cdot u^{p-1, k-p}_l \right) \\ 
= (z- \ol{z})u^{p-1, k-p}_l \wedge u_l^{p, k+1-p} + (u_l^{p, k+1-p} + z u_l^{p-1, k-p}) \wedge w_l \\
 = (z-\ol{z}) \left(u^{p-1, k-p}_l \wedge u_l^{p, k+1-p} + O(y^{-1}) \right) + zu_l^{p-1, k-p} \wedge w_{l}   
\end{multline*} 
and we can also ignore the term with $w_l= w_{l, k+1-p}$ by the same reason as before (the wedge product misses an element of $I^{p, k+1-p}$). 
By this and (\ref{eq:phi_z}), we finally obtain 
\begin{multline*}
\Phi(z) = \pm (z-\ol{z})^{h^{p-1,k-p}} \left( 1+ O\left(y^{-1} \right) \right) \left(\bigwedge_{i=0, i \neq p-1}^{k-1}  \bigwedge_l  u_l^{i,k-1-i} \right) \wedge \left( \bigwedge_{i=0}^k \bigwedge_l u^{i,k-i}_l \right) \\
 \wedge \bigwedge_{i=p+1}^k \bigwedge_l u_l^{i, k+1-i} \wedge \bigwedge_{i=k+2-p}^k \bigwedge_l u^{k+1-i, i}_l  
\wedge  \bigwedge_l \left(u_l^{p-1, k-p} \wedge u_l^{p, k+1-p} \right) +O\left(y^{-1} \right) \\
= \pm (z-\ol{z})^{h^{p-1,k-p}} \left( 1+ O\left(y^{-1} \right) \right) \left(\bigwedge_{i=0}^{k-1}  \bigwedge_l  u_l^{i,k-1-i} \right) \wedge \left( \bigwedge_{i=0}^k \bigwedge_l u^{i,k-i}_l \right)  \\ 
 \wedge \left( \bigwedge_{i=1}^k \bigwedge_l u_l^{i, k+1-i}  \right) +O \left(y^{-1} \right), 
\end{multline*}
thus this is nonzero when $y= \Image z \gg 0$.

\end{proof}

\begin{proof}[Proof of Theorem \ref{thm:intro2}]
We see that $N_n^{[1]}$ is an isomorphism by Proposition \ref{prop:H3CYdesc}. 
Then the latter statement follows from Theorem \ref{thm:main}. 
\end{proof}

Let us show Corollary \ref{cor:Tyurincase}.  

\begin{proof}[Proof of Corollary \ref{cor:Tyurincase}]
\noindent(i) By Remark \ref{rem:LefschetzCompati}, it is enough to check that $\Image \rho_2 \subset H^2(D, \bC)$ is compatible with the Lefschetz decomposition for some 
ample class $L$ on $D$. 

Let $L_1$ be an ample class on $X_1$ and $L:= L_1|_D$ be its restriction to $D$. 
Then we see that $L \in  \Image \rho_2$ and can check that $\Image \rho_2 = \bC \cdot L \oplus (L^{\perp} \cap \Image \rho_2)$. 
Hence $\Image \rho_2$ is compatible with the Lefschetz decomposition and obtain the claim. 

\noindent(ii) By (i), we see that $H^3$ of the examples in \cite{MR4584262} and \cite{Sano:2022aa} admits a pure Hodge structure 
since they are constructed as a semistable smoothing of an SNC variety as in the hypothesis. 
Then the $\partial\ol{\partial}$-lemma holds by \cite[Corollary 1.6]{MR4085665} since we have $H^i(X, \cO_X) = 0 = H^0(X, \Omega_X^i)$ for $i=1,2$ 
on the non-K\"{a}hler Calabi--Yau 3-folds $X$ as in \cite{MR4584262} and \cite{Sano:2022aa}. 
\end{proof}

\begin{rem}
Let $ \cX \to \Delta$ be a semistable smoothing of a proper SNC variety $X$ of dimension $n$.  
When $n=2$, then the Hodge filtration induces a pure Hodge structure on $H^2(\cX_t)$ by Theorem \ref{thm:hodgejunbi} (iv) and \cite[Lemma 1.5]{MR4085665}.  

Hence the problem of the Hodge symmetry on $H^n(X_t, \bC)$ makes sense when $n \ge 3$.  
It would be interesting to find an example of a semistable smoothing of a proper SNC variety on which the Hodge symmetry does not hold on the middle cohomology $H^n$. 

\end{rem}

\subsection{Positive definiteness of bilinear forms}

Before we prove Theorem \ref{thm:Hodge-RiemannH^3}, we explain the condition (*) in Theorem \ref{thm:Hodge-RiemannH^3}.

\begin{rem}\label{rem:condition*explain}
Let us explain the condition (*) in Theorem \ref{thm:Hodge-RiemannH^3}. 
Let 
\[
(\Ker \rho_3/ \Image \gamma_3)_{\bC}:= (\Ker \rho_3/ \Image \gamma_3) \otimes_{\bQ} \bC.
\] 
Since $\rho_3$ and $\gamma_3$ are morphism of Hodge structures, this admits a pure Hodge structure of weight $3$ 
and its $(2,1)$-part admits a pairing $Q_{W_3}$ defined by  
\[
Q_{W_3} (\eta, \ol{\eta}):= -\sqrt{-1} \langle \eta, \ol{\eta} \rangle
\]
for $\eta \in (\Ker \rho_3/ \Image \gamma_3)_{\bC}^{2,1}$ induced by the cup product $\langle \bullet, \bullet \rangle$ on $H^3(X^{(1)}, \bC)$. 
The pairing descends to $\Ker \rho_3/ \Image \gamma_3$ by the adjoint property of $\rho$ and $\gamma$. 

Similarly, let 
\[
(\Ker \gamma_4)_{\bC}:= \Ker \gamma_4 \otimes_{\bQ} \bC \subset H^2(Y^{(2)}, \bC).
\] 
Then $(\Ker \gamma_4)_{\bC}$ admits a pure Hodge structure of weight $2$ and its $(1,1)$-part admits a pairing $Q_{W_2}$ defined by 
\[
Q_{W_2}(\xi, \ol{\xi}):= - \langle \xi, \ol{\xi} \rangle
\] 
for $\xi \in (\Ker \gamma_4)_{\bC}^{1,1}$ induced by the cup product $\langle \bullet, \bullet \rangle$ on $H^2(X^{(2)}, \bC)$. 

Hence the condition (*) in Theorem \ref{thm:Hodge-RiemannH^3} makes sense. 
\end{rem}

\begin{rem}\label{rem:definitecondition}
Let us discuss when the above condition (*) holds. It holds when $(\Ker \rho_3/ \Image \gamma_3)^{2,1}_{\bC}$ and $(\Ker \gamma_4)^{1,1}$ are primitive. 

The pairing $Q_{W_3}$ on $(\Ker \rho_3/ \Image \gamma_3)^{2,1}_{\bC}$ is positive definite if $(\Image \gamma_3)^{2,1}$ contains some subspace $V_1$ 
such that  $\dim V_1 = h^1(X^{(1)}, \cO)$ and $- \sqrt{-1} \langle \bullet , \bullet \rangle$ is negative definite on $V_1$. 
This holds if $H^1(X^{(1)}, \cO) =0$, for example.

Since we have $\Ker \gamma_4 = (\Image \rho_2)^{\perp}(-1)$, we see that $Q_{W_2}$ is positive definite if $(\Image \rho_2)^{1,1}$ 
contains a subspace $V_2$ such that $\dim V_2 = h^0(X^{(2)}, \bC)$ and the cup product is positive definite on $V_2$. 
\end{rem}

We also use the following fact of linear algebra. 

\begin{lem}\label{lem:bilinearortho}
Let $(V, \langle \bullet , \bullet \rangle)$ be a finite dimensional $\bC$-vector space with a non-degenerate (symmetric) bilinear form. Let $W \subset V$ be a linear subspace and $W^{\perp}$ be its orthogonal complement. Then we have;
\begin{enumerate}
\item[(i)] $\dim W + \dim W^{\perp} = \dim V$.  
\item[(ii)] $W= (W^{\perp})^{\perp}$. 
\end{enumerate} 
\end{lem}

\begin{proof}
(i) is standard since the bilinear form induces an isomorphism $V \to V^*$ to the dual space and $W^{\perp}$ is the pull-back of $\{f \in V^* \mid f|_W =0 \}$. 
(ii) follows from $W \subset (W^{\perp})^{\perp}$ and $\dim (W^{\perp})^{\perp} = \dim V - \dim W^{\perp} = \dim W$.  
 \end{proof}

\begin{proof}[Proof of Theorem \ref{thm:Hodge-RiemannH^3}] 
Since the cup product is definite on $((\Image \rho_2)^\perp)^{1,1}$ and the $(2,0)$-part and the $(0,2)$-part of $H^2(Y^{(2)})$ are primitive, 
we see that the cup product is non-degenerate on $(\Image \rho_2)^{\perp}$, thus non-degenerate on $\Image \rho_2 = ((\Image \rho_2)^{\perp})^{\perp}$ as well by Lemma \ref{lem:bilinearortho}. Hence we see that $H^3(Y_s, \bC)$ admits a pure Hodge structure by Theorem \ref{thm:intro2}. 

Let $\cF^3, \cF^2 \subset \cH^3$ be the holomorphic subbundles as in Theorem \ref{thm:hodgejunbi}(iv). 
We have 
\[
F^3= I^{3,0} \oplus I^{3,1} \subset F^2 = I^{2,0} \oplus I^{2,1} \oplus I^{3,0} \oplus I^{2,2} \oplus I^{3,1}, 
\]
as in Proposition \ref{prop:limitMHSIpq}(ii). 
In order to show the positive definiteness as in Theorem \ref{thm:Hodge-RiemannH^3}, 
it is enough to calculate the signature of the form 
\[
\Phi_s \colon F^2_s \to \bR; \ \ \ \  \eta \mapsto Q(\eta, \bar{\eta})  
\]
for a small $s \in \Delta^*$. 
We know that the form on $F^3_s$ is negative definite by $\sqrt{-1}\langle \eta, \bar{\eta} \rangle >0$ for 
$0 \neq \eta \in H^0(Y_s, \Omega^3_{Y_s})$. 
Hence it is enough to find a subspace $W_s \subset F^2_s$ such that $\Phi_s|_{W_s}$ is positive definite and 
$\dim W_s = \dim F^2_s/ F^3_s$. 

Let $u_{\alpha}^{i,j} \in I^{i,j}$ ($\alpha =1,\ldots ,h^{i,j}$) be a $\bC$-basis for $I^{i,j} \subset F^2$ as in the proof of Theorem \ref{thm:main}. 
We choose $u_{\alpha}^{i,j}$'s so that the following holds. 


\begin{claim}\label{claim:ualphabetaproducts}
\begin{enumerate}
\item[(i)]We may take $u_{\alpha}^{2,1}$'s and $u_{\alpha}^{3,0}$'s so that 
\[
Q(u^{2,1}_{\alpha}, \ol{u^{2,1}_{\beta}}) = \delta_{\alpha \beta},  \ \ Q(u^{3,0}_{\alpha}, \ol{u^{3,0}_{\beta}}) = -\delta_{\alpha \beta}
\]
where $\delta_{\alpha \beta}$ is the Kronecker delta. 
\item[(ii)] We may also take $u_{\alpha}^{2,2}$'s and $u_{\alpha}^{1,3}$'s so that 
\[
Q(u_{\alpha}^{1,1}, \ol{u_{\beta}^{2,2}}) = \delta_{\alpha \beta}, \ \  Q(u_{\alpha}^{2,0}, \ol{u_{\beta}^{3,1}}) = -\delta_{\alpha \beta}. 
\]
\item[(iii)] We also have 
$
Q(u^{i,j}_{\alpha}, \ol{u^{k,l}_{\beta}}) = 0 
$ unless $(k,l)=(3-j,3-i)$. 
\end{enumerate}
\end{claim}

\begin{proof}[Proof of Claim] We use the notations and results in \cite{MR3395244}. 
Let $((A_{\bQ}, W), (A_{\bC}, W, F), \alpha)$ be the cohomological mixed Hodge complex as in \cite[Definition 5.18]{MR3395244} such that 
$H^k(Y, \Lambda_Y^{\bullet}) \simeq H^k (A_{\bC})$ for $k \in \bZ$.  


\noindent(i) By Proposition  \ref{prop:H3CYdesc}(ii), we have   
\[
\Gr^W_3 H^3(Y, \Lambda_Y^{\bullet}) \simeq (\Ker \rho_3/ \Image \gamma_3) \otimes_{\bQ} {\bC} 
\] 
Note that the pairing on this is induced by the pairing on $H^3(Y, \Gr^W_0 A_{\bC})$, where 
\[
\Gr^W_0 A_{\bC} \simeq \bigoplus_{i=1}^l \Omega^{\bullet}_{Y_i} 
\]
by \cite[(5.22.2)]{MR3395244}.
By \cite[Lemma 6.13 and the proof of Theorem 8.11]{MR3395244}, we see that the pairing is induced by the cup product on $H^3(Y_i, \bC)$, thus we obtain the claim by the positive definiteness of the pairing $Q_{W_3}$ on $(\Ker \rho_3/ \Image \gamma_3)_{\bC}^{2,1}$ (Note that $(\Ker \rho_3/ \Image \gamma_3)_{\bC}^{3,0}$ is primitive and a similar form on $(\Ker \rho_3/ \Image \gamma_3)_{\bC}^{3,0}$ is definite).  


\noindent(ii) By \cite[(5.22.2)]{MR3395244}, we obtain  
\[
\Gr^W_1 A_{\bC} \simeq \bigoplus_{1 \le i<j \le l} \Omega^{\bullet}_{Y_{ij}}[-1] \simeq \Gr^W_{-1} A_{\bC},  
\]   
where $Y_{ij}= Y_i \cap Y_j$. 
Recall that, by Proposition \ref{prop:H3CYdesc}(ii), we have  
\[
W_4/W_3 \simeq \Ker \gamma_4, \ \ W_2 \simeq \Coker \rho_2. 
\]
We see that the pairing on $W_4/W_3 \times W_2$ is induced by the pairing on $H^3(\Gr^W_1 A_{\bC}) \times H^3( \Gr^W_{-1} A_{\bC})$ which is induced by the cup products on $H^2(Y_{ij}, \bC)$ by \cite[Lemma 6.13 and the proof of Theorem 8.11]{MR3395244}. 
Note that the monodromy logarithm $N \colon H^3(\Gr^W_1 A_{\bC}) \to H^3( \Gr^W_{-1} A_{\bC})$ is induced by  
$2\pi \sqrt{-1} \id $ on $H^2(Y_{ij})$ since $N_A$ in \cite[5.23]{MR3395244} is induced by $2 \pi \sqrt{-1} \pi_r$, 
where $\pi_r \colon \omega_Y^{p+1}/W_r \to \omega_Y^{p+1}/W_{r+1}$ is the projection for the log de Rham complex $\omega_Y^{\bullet}$ and the weight filtration $W_{\bullet}$ on it.  
By this and the definiteness of $Q_{W_2}$ on $(\Ker \gamma_4)^{1,1}$, 
we obtain the former claim. The latter claim follows since $(\Ker \gamma_4)^{2,0}$ is primitive.  



\noindent(iii) This follows from the argument as in \cite[Proof of (4) before Lemma 6.1]{Chen:2024aa}. 
\end{proof}

Let $\varphi \colon \fH \to \Delta^*; z \mapsto s=\exp (2\pi \sqrt{-1}z)$ be the universal cover. 
Let $U_{\alpha}^{i,j}(z) \colon \mathfrak{H} \to H^3$ be the function corresponding to the holomorphic section $U_{\alpha}^{i,j}(s) \in \cF^2$ which is a lift of $u_{\alpha}^{i,j} \in F^2$ as in the proof of Theorem \ref{thm:main}. 

We consider a small $s \in \Delta^*$ and assume that $z= x+ \sqrt{-1} y$ ($x, y \in \bR$) satisfies that $x$ lies in a bounded strip and $y$ is sufficiently large. 
Let 
\[
V_{\alpha}^{2,0}(z):= U_{\alpha}^{2,0}(z) -2z U_{\alpha}^{3,1}(z) \in (\varphi^* \cF^2)_z \simeq F^2_s. 
\] 
Let $W_s \subset F^2_s$ be the subspace generated by 
\begin{multline}
B(z):= 
 \{V_{\alpha}^{2,0}(z) \mid \alpha = 1, \ldots , h^{2,0} \} \cup \{U_{\alpha}^{2,1}(z) \mid \alpha =1,\ldots , h^{2,1} \} \\
  \cup 
\{U_{\alpha}^{2,2}(z) \mid \alpha = 1, \ldots , h^{2,2} \}. 
\end{multline}
We can check $\dim W_s = \dim F^2_s/ F^3_s$. 


Let $u_{\alpha}^{i,j}(z)= e^{zN} \cdot u_{\alpha}^{i,j} \in \varphi^{*} \cH^3|_{\Delta^*}$  be the $H^3$-valued function determined by $u_{\alpha}^{i,j} \in H^3$. Then we have 
\[
u_{\alpha}^{2,0}(z) = u_{\alpha}^{2,0}, \ \ 
u_{\alpha}^{2,1}(z) = u_{\alpha}^{2,1}, 
\]
\[ 
u_{\alpha}^{2,2}(z) = u_{\alpha}^{2,2} + z u_{\alpha}^{1,1}, \ \ u_{\alpha}^{3,1}(z) = u^{3,1}_{\alpha} + z u^{2,0}_{\alpha},
\]
where $u_{\alpha}^{1,1}= N(u_{\alpha}^{2,2})$ and $u_{\alpha}^{2,0}= N(u_{\alpha}^{3,1})$, respectively. 
Then we may write 
\[
U_{\alpha}^{2,1}(z) = \sum A_{\beta}^{i,j} u_{\beta}^{i,j} (z), \ \ U_{\alpha}^{2,2}(z) = \sum B_{\beta}^{i,j} u_{\beta}^{i,j} (z), 
\]
\[
 V_{\alpha}^{2,0}(z) = \sum (C_{1,\beta}^{i,j}-2z C_{2, \beta}^{i,j}) u_{\beta}^{i,j} (z), 
\]
for some holomorphic functions $A_{\beta}^{i,j}, B_{\beta}^{i,j}, C_{1, \beta}^{i,j}, C_{2, \beta}^{i,j}$ on $\Delta$ and we have 
\begin{equation}\label{eq:U^21z}
U_{\alpha}^{2,1}(z) = u_{\alpha}^{2,1} + O\left(y^{-1} \right), 
\end{equation}
\begin{equation}\label{eq:U^22U^13z}
 U_{\alpha}^{2,2}(z) = u_{\alpha}^{2,2}+ z u_{\alpha}^{1,1} + O\left(y^{-1} \right), \ \ 
V_{\alpha}^{2,0}(z) = u_{\alpha}^{3,1} - z u_{\alpha}^{2,0} + O\left(y^{-1} \right), 
\end{equation}
where $y:= \Image z$. (In fact, since 
\[
s= \exp (2\pi \sqrt{-1}(x + \sqrt{-1}y)) = \frac{1}{e^{2\pi y}} \exp(2\pi \sqrt{-1}x),
\] 
the above $O(y^{-1})$'s are also $O(e^{-y})$). 

We consider the intersection matrix of $Q(\bullet, \bullet)$ for the basis $B(z)$. 

Let 
\[
Q_{11}:= \begin{pmatrix} Q\left( U_{\alpha}^{2,1} (z), \ol{U_{\beta}^{2,1}(z)} \right) \end{pmatrix}_{1\le \alpha, \beta \le h^{2,1}}, 
\]
\[
Q_{22}:= \begin{pmatrix} Q\left(U_{\alpha}^{2,2} (z), \ol{U_{\beta}^{2,2}(z)} \right) \end{pmatrix}_{1\le \alpha, \beta \le h^{2,2}}, 
\]
\[
Q_{33}:= \begin{pmatrix} Q \left( V_{\alpha}^{2,0}(z), \ol{V_{\beta}^{2,0}(z)} \right) \end{pmatrix}_{1\le \alpha, \beta \le h^{2,0}}, 
\]
\[ 
Q_{12}:= \begin{pmatrix} Q \left( U_{\alpha}^{2,1} (z), \ol{U_{\beta}^{2,2}(z)} \right) \end{pmatrix}_{\substack{1 \le \alpha \le h^{2,1}, \\ 1\le \beta \le h^{2,2} }}, 
\]
\[
Q_{13}:= \begin{pmatrix} Q \left( U_{\alpha}^{2,1} (z), \ol{V_{\beta}^{2,0}(z)} \right) \end{pmatrix}_{\substack{1 \le \alpha \le h^{2,1}, \\ 1\le \beta \le h^{2,0} }}, 
\] 
\[
Q_{23}:= \begin{pmatrix} Q \left( U_{\alpha}^{2,2} (z), \ol{V_{\beta}^{2,0}(z)} \right) \end{pmatrix}_{\substack{1 \le \alpha \le h^{2,2}, \\ 1\le \beta \le h^{2,0} }}.  
\]
Then the associated matrix $Q$ for the quadratic form $Q(\bullet, \bullet)$ is 
\[
Q = \begin{pmatrix}
Q_{11} & Q_{12} & Q_{13} \\
Q_{21} & Q_{22} & Q_{23} \\ 
Q_{31} & Q_{32} & Q_{33}
\end{pmatrix}, 
\]
where $Q_{ji}:= Q_{ij}^*$ is the adjoint matrix for $1 \le i < j \le 3$. 

By the equation (\ref{eq:U^22U^13z}) and Claim \ref{claim:ualphabetaproducts}, we see that 
\begin{multline}\label{eq:Q22approximation}
Q \left( U_{\alpha}^{2,2}(z), \ol{U_{\beta}^{2,2}(z)} \right) = 
Q \left( u_{\alpha}^{2,2}+ z u_{\alpha}^{1,1} + O\left(y^{-1} \right), \ol{u_{\beta}^{2,2}} + \ol{z} \ol{u_{\beta}^{1,1}} + O\left(y^{-1} \right) \right)
\\ = z Q \left( u_{\alpha}^{1,1}, \ol{u_{\beta}^{2,2}} \right) + \ol{z} Q \left( u_{\alpha}^{2,2}, \ol{u_{\beta}^{1,1}} \right) + \Image (z) \cdot O\left(y^{-1} \right) 
= \begin{cases}
2 \Image (z) \left( 1 + O\left(y^{-1} \right) \right) & (\alpha = \beta) \\ 
\Image (z) \cdot O\left(y^{-1} \right) & ( \alpha \neq \beta)
\end{cases}. 
\end{multline}
Hence we obtain 
\[
Q_{22} = 2 \Image (z) \cdot \left( E_2 + O\left(y^{-1} \right) \right), 
\]
where $E_2$ is the identity matrix of size $h^{2,2}$. 

Similarly, by the equation (\ref{eq:U^22U^13z}), we see that 
\begin{multline}\label{eq:Q33approximation}
Q \left( V_{\alpha}^{2,0}(z), \ol{V_{\beta}^{2,0}(z)} \right) 
= Q \left( u_{\alpha}^{3,1} - z u_{\alpha}^{2,0} + O\left(y^{-1} \right) , \ol{u_{\beta}^{3,1}} - \ol{z} \ol{u_{\beta}^{2,0}} + O\left(y^{-1} \right) \right)  \\ 
= -z Q\left( u_{\alpha}^{2,0}, \ol{u_{\beta}^{3,1}} \right)- \ol{z} Q \left( u_{\alpha}^{3,1}, \ol{u_{\beta}^{2,0}} \right) + \Image (z) \cdot O\left(y^{-1} \right) 
= \begin{cases}
2 \Image (z) \left( 1 + O\left(y^{-1} \right) \right) & (\alpha = \beta) \\ 
\Image (z) \cdot O\left(y^{-1} \right) & ( \alpha \neq \beta)
\end{cases}, 
\end{multline}
thus we obtain  
\[
Q_{33} = 2\Image (z) \cdot \left( E_3 + O\left(y^{-1} \right) \right), 
\]
where $E_3$ is the identity matrix of size $h^{2,0}$. 

By Claim \ref{claim:ualphabetaproducts}, the equation (\ref{eq:U^21z}) and a computation as above in (\ref{eq:Q22approximation}), we have 
\[
Q_{11}(z) = E_{1} + O\left(y^{-1} \right) 
\]
for the identity matrix $E_1$ of size $h^{2,1}$. 
By similar calculations, we also see that
\[
Q_{12}(z) = O(y^{-1}), \ \ Q_{13}(z) = O(y^{-1}), \ \ Q_{23}(z) = O(y^{-1}). 
\]
Hence we have 
\[
Q= 
\begin{pmatrix} 
E_1 & O \\
O & 2y E_{23}
\end{pmatrix} + O(y^{-1}) \ \ (E_{23}:= E_2 \oplus E_3), 
\]
thus $Q$ is positive definite when $y \gg 0$. 
(If necessary, we can also use $\sqrt{2y} U^{2,1}_{\alpha}(z)$ instead of $U^{2,1}_{\alpha}(z)$ in the basis $B(z)$ so that $Q$ becomes $2yE + O(y^{-1})$ with respect to that choice of the basis.) 

Hence we obtain the required positive definiteness of the forms on $W_s$ and $H_s^{2,1}$. 
%
%
%
%
\end{proof}

By Theorem \ref{thm:Hodge-RiemannH^3} and Remark \ref{rem:definitecondition}, we obtain the following.

\begin{proof}[Proof of Corollary \ref{cor:Tyurinh^1=0case}]
The definiteness of $Q_{W_3}$ follows from $H^1(Y_j, \cO) =0$ by Remark \ref{rem:definitecondition}. 

Since $\Image \rho_2$ contains images of ample classes on $Y_j$, we obtain the definiteness of $Q_{W_2}$ by Remark \ref{rem:definitecondition} as well. 
\end{proof}

\begin{rem}
In Corollary \ref{cor:Tyurinh^1=0case}, the examples in \cite{MR4584262} are covered. 
We can check that the examples in \cite{Sano:2022aa} satisfies the conditions on the pairings which was already written in \cite{Chen:2024aa}. 
We can also check that the degeneration of Clemens manifolds as in Example \ref{eg:Clemensmfd} satisfies the condition (*).  
\end{rem}

\section{Examples}

\subsection{Remarks on Hodge--Riemann line bundles}	

Ample classes (or K\"{a}hler classes) satisfy the Hard Lefschetz property and the Hodge--Riemann bilinear relations. 
However, not only these are the classes with such properties. 

\begin{defn}
Let $X$ be a compact complex manifold of dimension $n$ with a pure Hodge structures on $H^k(X, \bC)$ for $k \in \bZ$ as in Proposition \ref{prop:pureHodgekopposed} and Remark \ref{rem:putativeHodge}(i).  
Let $\cL$ be a holomorphic line bundle on $X$. 
\begin{enumerate}
\item[(i)] We say that $\cL$ is {\it Lefschetz} if the operator
\[
\cup c_1(\cL)^i \colon H^{n-i}(X, \bQ) \rightarrow H^{n+i}(X, \bQ)
\] is an isomorphism 
for $i=1,\ldots ,n$. 
\item[(ii)]  We say that $\cL$ is {\it Hodge-Riemann} if $\cL$ is Lefschetz and the Hodge-Riemann relation holds for $\cL$, that is, 
on the primitive cohomology $P^k:=P^k_{\cL}:= \Ker (\cup c_1(\cL)^{n-k+1}) \subset  H^k(X, \bC)$ with the decomposition $P^k = \bigoplus P^{p,k-p}$, 
the pairing determined by 
\[
(\alpha, \beta):= (-1)^{\frac{1}{2}k(k-1)} \sqrt{-1}^{2p-k} \int  \alpha \cup \ol{\beta} \cup c_1(\cL)^{n-k}
\]
for $\alpha, \beta \in P^{p,k-p}$ 
is positive definite for $k \le n$. 
\end{enumerate}
\end{defn}

As in the following examples, there are Lefschetz or Hodge--Riemann line bundles which are not ample. 

\begin{eg}
\begin{enumerate}
\item[(i)] Let $\cL$ be an ample line bundle on a smooth projective variety $X$. 
Then it is well known that $\cL$ is Lefschetz by the hard Lefschetz theorem. 
On the other hand, its dual $\cL^{-1}$ is also Lefschetz. 
\item[(ii)] Let $X:= \bP^1 \times \bP^1$ and $\cL:= p_1^*\cO_{\bP^1}(a) \otimes p_2^*\cO_{\bP^1}(b)$ for the projections $p_1, p_2$. 
Then we can easily check that $\cL$ is Lefschetz if and only if $ab \neq 0$ since it is equivalent to $c_1(\cL)^2 \neq 0$. 
We can also check that $\cL$ is Hodge--Riemann if and only if $ab >0$. 
\item[(iii)] Let $\mu \colon \tilde{X} \to X$ be a blow-up along a smooth subvariety $Z \subset X$ of codimension $2$ and $\cL$ be an ample line bundle on $X$. 
Then it is known that $\mu^* L$ is a Hodge--Riemann line bundle (cf. \cite[Theorem 2.3.1]{MR1951443}). 
\item[(iv)] Let $X= \bP^1 \times \bP^1 \times \bP^1$ and let $\cL:= \cO_X(a_1 F_1 +a_2 F_2 + a_3 F_3)$, where $F_i:= p_i^* \cO_{\bP^1}(1)$ is the fiber class of 
the projection $p_i$ for $i=1,2,3$. We can compute that $\cL$ is Hodge--Riemann if and only if $a_1 a_2 a_3 >0$. For example, $-F_1 -F_2 + F_3$ is not ample, but Hodge--Riemann. 
\end{enumerate}
\end{eg}

Guillen--Navarro Aznar \cite[(5.2) Th\'{e}or\`{e}me]{MR1068383} proved that, for a K\"{a}hler semistable degeneration $\phi \colon \cX \to \Delta$,
 the monodromy nilpotent operator induces isomorphisms on the graded pieces of $H^k$. 
 The same proof shows the following. 

\begin{prop}\label{prop:HodgeRiemannMonodromyIsom}
Let $\cX \to \Delta$ be a semistable smoothing of a proper SNC variety $X = \bigcup X_i$ and $H^k=H^k(X, \Lambda^{\bullet}_{X})$ be as in 
Theorem \ref{thm:hodgejunbi}. 
Suppose that there exists a line bundle $\cL_0$ on $X$ such that
 $\cL_0|_{D}$ for all irreducible components $D$ of $X^{(l)} = \coprod_{i_1< \cdots < i_l} (X_{i_1} \cap \cdots \cap X_{i_l})$ are Hodge Riemann.  

Then, for $r \ge 0$, the monodromy nilpotent operator $N_k \colon H^k \rightarrow H^k$  
 induces an isomorphism 
 $
 N_k^{[r]} \colon W_{k+r}/ W_{k+r-1} \rightarrow W_{k-r}/ W_{k-r-1}.
 $  
\end{prop}	

\begin{proof}
The proof of the isomorphism is parallel to that of \cite[Theorem 11.40]{MR2393625} or \cite[(5.2) Th\'{e}or\`{e}me]{MR1068383} 
in which only the Lefschetz property of $\cL_0|_{D}$ and the positive definiteness of the pairings are essential.  
\end{proof}


Motivated by Proposition \ref{prop:HodgeRiemannMonodromyIsom}, we look for line bundles on SNC varieties whose 
restriction to irreducible components are Lefschetz or Hodge--Riemann in the following examples. 
	
\begin{eg}
Clemens \cite[pp.229]{MR444662} exhibited an example of a semistable degeneration $\cS \to \Delta$ of a Hopf surface 
such that $\cS_0$ is a normal crossing surface $\ol{\bF_1}$ which is constructed by identifying 
negative section $\sigma_0$ and a positive section $\sigma_{\infty}$ of the Hirzebruch surface $\bF_1$. 
Let $\ol{\sigma} \subset \cS_0$ be the singular locus. 

To make the computation easier, we make a birational modification so that the semistable degeneration has a SNC central fiber. 
Let $\beta \colon \Delta \to \Delta$ be a double cover defined by $\beta(t) = t^2$, $\cT:= \cS \times_{\Delta} \Delta$ be 
the fiber product by $\beta$ and $\mu \colon \tilde{\cT} \to \cT$ be the blow-up along $\beta_{\cT}^{-1} (\ol{\sigma}) \subset \cT$, 
where $\beta_{\cT} \colon \cT \to \cS$ is the base change of $\beta$. Then we have the commutative diagram 
\[
\xymatrix{
\tilde{\cT} \ar[r]^{\mu} \ar[rd]_{\phi_{\tilde{\cT}}} & \cT \ar[r]^{\beta_{\cT}} \ar[d]^{\phi_{\cT}} & \cS \ar[d]^{\phi_{\cS}} \\
 & \Delta \ar[r]^{\beta} & \Delta
}. 
\]
We see that $\phi_{\tilde{\cT}} \colon \tilde{\cT} \to \Delta$ is semistable and $\tilde{\cT}_0 = \tilde{T} \cup E$ is an SNC surface 
such that $\tilde{T} \simeq E \simeq \bF_1$. (Indeed, $\cT$ has $A_1$-singularities along $\beta_{\cT}^{-1} (\ol{\sigma})$, 
thus its blow-up induces a semistable degeneration $\tilde{\cT} \to \Delta$. ) 
We also see that $\tilde{T} \cap E = C_1 \cup C_2$ and, if $C_1 \subset \tilde{T}$ is a negative section with $C_1^2 =-1$ on $T$, 
then $C_1 \subset E$ is a positive section with $C_1^2 =1$ on $E$.  
($C_2$ is the positive section on $\tilde{T}$ and the negative section on $E$.) 

\begin{claim}\label{claim:LefschetzHopf}
There exists $\cL_0 \in \Pic \tilde{\cT}_0$ such that $\cL_1:= \cL_0|_{\tilde{T}} \in \Pic \tilde{T}$ and $\cL_2:= \cL_0|_{E} \in \Pic E$  are Lefschetz and $\cL_1|_{\tilde{T} \cap E} = \cL_2|_{\tilde{T} \cap E}$ is ample.  
\end{claim} 

\begin{proof}[Proof of Claim]
Let $\cL_1 = a_1 h_1 + a_2 f_1 \in \Pic T$, where $h_1, f_1 \in \Pic T \simeq \Pic \bF_1$ are the negative section and the fiber of $\bF_1$. 
Let $ \cL_2 = b_1 h_2 + b_2 f_2 \in \Pic E$, where $h_2, f_2 \in \Pic E \simeq \Pic \bF_1$ are the negative section and the fiber of $\bF_1$. 
Note that $\cL_1$ and $\cL_2$ can be glued to $\cL_0 \in \tilde{T}_0$ if and only if $\cL_1 \cdot C_i = \cL_2 \cdot C_i$ for $i=1,2$.  
We have 
\begin{multline}\label{eq:L_iC_jintersection}
\cL_1 \cdot C_1 = (a_1 h_1 + a_2 f_1) \cdot h_1 = a_2 -a_1, \ \ \cL_2 \cdot C_1 = (b_1 h_2 + b_2 f_2) \cdot (h_2 + f_2) = b_2, \\ 
\cL_1 \cdot C_2 = (a_1 h_1 + a_2 f_1) \cdot (h_1+ f_1) = a_2, \ \ \cL_2 \cdot C_2 = (b_1 h_2 + b_2 f_2) \cdot (h_2 ) = b_2-b_1. 
\end{multline}
Hence $\cL_1= a_1 h_1 + a_2 f_1$ can be glued with $\cL_2 = (-a_1) h_2 + (a_2 -a_1) f_2$ for all $a_1, a_2 \in \bZ$. 
We see that 
\begin{equation}\label{eq:L_i^2}
\cL_1^2 = a_1(-a_1 + 2a_2), \ \ \cL_2^2 = a_1 (a_1-2a_2).
\end{equation} 
Note that $\cL_i$ is Lefschetz if and only if $\cL_i^2  \neq 0$. 
Note also that $\cL_i|_{C_j}$ is ample for $j=1,2$ if and only if $a_2-a_1>0$ and $a_2>0$ by the above equation (\ref{eq:L_iC_jintersection}).  
Hence, if we let $(a_1, a_2) = (a, 2a)$ for $a>0$ for example, then the induced line bundles 
 $\cL_1 = a(h_1 +2f_1)$ and $\cL_2 = a(-h_2 + f_2)$ satisfy the Lefschetz property and 
 induce an ample line bundle on $\tilde{T} \cap E$. 
\end {proof}

Claim \ref{claim:LefschetzHopf} shows that, in Proposition \ref{prop:HodgeRiemannMonodromyIsom}, requiring the Lefschetz assumption and the ampleness on the intersection are not sufficient.  
Indeed, we can check that, if there exists $\cL_1$ and $\cL_2$ as in Claim \ref{claim:LefschetzHopf}, then $\cL_1$ or $\cL_2$ is not Hodge-Riemann since $\cL_1^2 = - \cL_2^2$ by (\ref{eq:L_i^2}). 

Moreover, since $b_2 (\cS_t) =0$ on a Hopf surface $\cS_t$, we see that $\cL_0 \in \Pic \tilde{T}_0$ induced by the pair $(\cL_1, \cL_2) \in \Pic \tilde{T} \times \Pic E$ is sent to $0$ in $H^2$ of $\tilde{\cT}_s \simeq \cS_t$ (for a small $s \in \Delta^*$ and $t=s^2$) by the composition 
\[
\Pic \tilde{\cT}_0 \to H^2(\tilde{\cT}_0, \bZ) \simeq H^2(\tilde{\cT}, \bZ) \to H^2(\cS_t, \bZ)=0,   
\]
thus it does not induce a ``Lefschetz class'' in $H^2(\cS_t)$. 
\end{eg}
	
\begin{eg}\label{eg:HashimotoSano}
It would be nice if we can exhibit a non-projective SNC Calabi-Yau variety $X = X_1 \cup X_2$ with irreducible components $X_1, X_2$ and a line bundle $\cL_0$ such that 
$\cL_i:= \cL_0|_{X_i}$ for $i=1,2$ and $\cL_{12}:= \cL_0|_{X_{12}}$ on $X_{12}:= X_1 \cap X_2$ are all Hodge--Riemann. 
Although we could not find such line bundles so far, 
we can give an example of a Lefschetz line bundle $\cL_t$ on a non-K\"{a}hler Calabi--Yau 3-fold $X_t$ constructed in \cite{MR4584262} as follows. 

Let $D \in |{-}K_{\bP^1 \times \bP^1 \times \bP^1}|$ be a general member and $\iota \in \Aut D$ an automorphism of infinite order as in \cite{MR4584262}. 
Let $f_j \subset D$ for $j=1,2,3$ be the fiber of the $j$-th projection $D \to \bP^1$. 
Let $a \in \bZ_{>0}$ and $\mu \colon X_1 \to \bP^1 \times \bP^1 \times \bP^1$ be the blow-up along smooth curves 
$$C_1, \ldots , C_a, C_{a+1} \subset D, $$
where $C_1, \ldots , C_a \in |f_1|$ are the disjoint smooth curves and 
\[
C_{a+1} \in |2 \left(f_1+f_2 +f_3 + (\iota^a)^*(f_1+f_2 +f_3)\right)  - af_1 |
\] is a smooth curve 
 as in \cite{MR4584262} and $X_2= \bP^1 \times \bP^1 \times \bP^1$. Let $X:= X_1 \cup X_2$ be a d-semistable SNC Calabi--Yau variety which is determined by 
the isomorphism $\iota^a \circ \mu|_{\tilde{D}_1} \colon \tilde{D}_1 \to D_2$ from the strict transform $\tilde{D}_1 \subset X_1$ of $D$ to $D=D_2 \subset X_2$. 
Let $\cL_0 \in \Pic X$ be the line bundle induced by 
\[
\cL_1 = \mu^* \cO(F_1+ F_2 + F_3), \ \ \cL_2= \cO( (8a^2+1)F_1+ (1+4a) F_2 + (1-4a) F_3),  
\]
where $F_j \subset \bP^1 \times \bP^1 \times \bP^1$ for $j=1,2,3$ be the fiber of the $j$-th projection $ \pi_j \colon \bP^1 \times \bP^1 \times \bP^1 \to \bP^1$. 
Let $\cX \to \Delta$ be a semistable smoothing of $X$ and $X_t$ be its general smooth fiber. 
Let $\cL_t \in \Pic X_t$ be the induced line bundle by the isomorphism $\Pic \cX \simeq \Pic X$ and the restriction $\Pic \cX \to \Pic X_t$. 
By the weight spectral sequence, we have the commutative diagram 
\[
\xymatrix{
H^0(D) \ar[r]^{\gamma_2} \ar[d]^{\cup \cL_D} & \bigoplus_{i=1}^2 H^2(X_i) \ar[r]^{\rho_2} \ar[d]^{(\cup \cL_1, \cup \cL_2)= \phi_{\cL}} & H^2(D) \ar[d]^{\cup \cL_D} \\
H^2(D) \ar[r]^{\gamma_4} & \bigoplus_{i=1}^2 H^4(X_i) \ar[r]^{\rho_4} & H^4(D),  
}
\]
where the horizontal sequences are complexes and its cohomology groups in the middle are $H^2(X_t)$ and $H^4(X_t)$. 
 \begin{claim} 
  The homomorphism $\phi_{\cL_t}:=\cup \cL_t \colon H^2(X_t) \to H^4(X_t)$ is an isomorphism. 
  Hence we see that $\cL_t \in \Pic X_t$ is Lefschetz. 
\end{claim}

\begin{proof}[Proof of Claim]
Note that we have  
\[
H^2(X_t) \simeq \Ker \rho_2 / \Image \gamma_2, \ \ H^4(X_t)\simeq \Ker \rho_4/ \Image \gamma_4. 
\] 
We shall show that the homomorphism $$\psi_{\cL} \colon  \Ker \rho_2 / \Image \gamma_2 \to \Ker \rho_4/ \Image \gamma_4$$ 
induced by $\phi_{\cL}$ is injective, thus isomorphism by the dimension count.  

Let $E_i \subset X_1$ for $i=1,\ldots ,a+1$ be the  $\mu$-exceptional divisors over $C_i$. 
 We can take the basis of $\Ker \rho_2 / \Image \gamma_2$ induced by the $(a+3)$ elements 
 $\tilde{F}_1, \tilde{F}_2, \tilde{F}_3, \tilde{E}_1, \ldots , \tilde{E}_a \in \Ker \rho_2$ 
 given by 
\[
\tilde{F}_1 = (\mu^* F_1, F_1), 
\]
\[
\tilde{F}_2 = (\mu^* \left( (4a^2 - 2a)F_1 +(1-2a)F_2 + 2a F_3 \right) , F_2) 
\]
\[
\tilde{F}_3 = (\mu^* \left( (4a^2 + 2a)F_1 -2a F_2 + (1+2a) F_3 \right), F_3), 
\]
\[
\tilde{E}_i = (\cO(E_i), \cO(F_1))  \ \ (i=1, \ldots ,a). 
\]
For $1 \le i < j \le 3$, let $\pi_{ij}= \pi_i \times \pi_j \colon \bP^1 \times \bP^1 \times \bP^1 \to \bP^1 \times \bP^1$.  
Let $f_{ij}:= [\pi_{ij}^{-1} (p)] \in H^4(X_2)$ for $p \in \bP^1 \times \bP^1$ and $1\le i < j \le 3$ be the fiber class. 
Let $e_i:= [\mu^{-1}(q_i)] \in H^4(X_1)$ be the class of the fibers of $\mu$ over $q_i \in C_i$ for $i=1,\ldots a+1$. 
We know that $H^4(X_1)$ are generated by $\mu^*(f_{ij})$'s and $e_i$'s and $H^4(X_2)$ is generated by $f_{ij}$'s.  

We can check the injectivity of $\psi_{\cL}$ as follows. Note that    
\[
\phi_{\cL}(\tilde{F}_1) = (\mu^*(f_{12}+ f_{13}), (1+4a) f_{12} + (1-4a)f_{13}), 
\]
\[
\phi_{\cL} (\tilde{F}_2) = \left((4a^2-4a+1)f_{12}+ 4a^2 f_{13} + f_{23}, (8a^2+1) f_{12} + (1-4a) f_{23} \right), 
\]
\[
\phi_{\cL} (\tilde{F}_3) = \left( 4a^2 f_{12} + (4a^2+4a+1) f_{13} +  f_{23}, (8a^2+1) f_{13} + (1+4a) f_{23} \right). 
\] 
Since the image of $\gamma_4$ is generated by $(\mu^*f_{ij}, -f_{ij})$ for $1 \le i < j \le 3$, we see that 
\[
\tau \colon H^4(X_1) \oplus H^4(X_2) \to H^4(X_1) ; (\mu^* (\alpha) + \sum a_i e_i, \beta) \mapsto \mu^*(\alpha + \beta) + \sum a_i e_i
\]
induces an isomorphism $\bar{\tau} \colon \Coker \gamma_4 \to H^4(X_1)$. 
Since we have 
\[
\tau (\phi_{\cL}(\tilde{F}_1)) = \mu^*((2+4a)f_{12} + (2-4a)f_{13}), 
\]
\[
\tau (\phi_{\cL}(\tilde{F}_2)) = \mu^*((12a^2-4a+2)f_{12} + 4a^2 f_{13} + (2-4a)f_{23}),  
\]
\[
\tau (\phi_{\cL}(\tilde{F}_3)) = \mu^*((4a^2)f_{12} + (12a^2+4a+2) f_{13} + (2+4a)f_{23}),   
\]
and 
\[
\begin{vmatrix}
2+ 4a & 12 a^2 -4a+2 & 4a^2 \\
2-4a & 4a^2 & 12a^2 +4a +2 \\
0 & 2-4a & 2+4a 
\end{vmatrix} = 8(64a^4-2) \neq 0, 
\]
we see that $\tau (\phi_{\cL}(\tilde{F}_1)), \tau (\phi_{\cL}(\tilde{F}_2)), \tau (\phi_{\cL}(\tilde{F}_3))$ are linearly independent. 
By this and the description 
\[
\phi_{\cL} ( \tilde{E}_j ) = \left( 4 e_j, (1+4a) f_{12} + (1-4a) f_{13} \right) \ \ (j=1, \ldots ,a), 
\] we also see that $\tau (\phi_{\cL}(\tilde{F}_1)), \tau (\phi_{\cL}(\tilde{F}_2)), \tau (\phi_{\cL}(\tilde{F}_3)), \tau (\phi_{\cL}(\tilde{E}_1)), \ldots , \tau (\phi_{\cL}(\tilde{E}_a))$ are linearly independent. 
This implies the injectivity of $\psi_{\cL}$.  
\end{proof}

Hence we see that $\cL_t$ is Lefschetz. 
However, we also have the following.

\begin{claim}\label{claim:L_tpmnotHodgeRiemann}
 $\cL_t$ and $\cL_t^{-1}$ are not Hodge--Riemann. 
 \end{claim} 

\begin{proof}
We calculate that 
\begin{multline}
\cL_t^3 = \cL_1^3 + \cL_2^3 = (F_1+ F_2 + F_3)^3 + ((8a^2+1)F_1+ (1+4a) F_2 + (1-4a) F_3)^3  \\ 
= 6 \left(1+ (8a^2+1)(1+4a)(1-4a) \right)<0, 
\end{multline}
thus $\cL_t$ is not Hodge--Riemann. 
We can also check that $\cL_t^{-1}$ is not Hodge--Riemann by calculating $(\cL_t^2)^{\perp} \subset H^2(X_t)$. 
Since we have  
\[
(\cL_1^2, \cL_2^2) = (2(F_{12}+ F_{23} + F_{13}), 2((8a^2+1)(1+4a)F_{12} + (8a^2+1)(1-4a)F_{13} + (1-16a^2) F_{23})), 
\]
the element
\[
\Delta_{21}:= \left( \mu^*(F_2 - F_1), (4a^2 + 2a -1) F_1 + (1+2a)F_2 -2a F_3 \right) \in H^2(X)
\]
defines an element $[\Delta_{21}] \in H^2(X_t)$. Then we can calculate 
\begin{multline}
\Delta_{21}^2 \cdot \cL_t = \mu^*(F_2-F_1)^2 \cdot \mu^*(F_1 + F_2 + F_3) \\ 
+ \left((4a^2 + 2a -1) F_1 + (1+2a)F_2 -2a F_3 \right)^2 \cdot \left( (8a^2+1)F_1+ (1+4a) F_2 + (1-4a) F_3 \right) \\
 = -2 + 2 (-32a^4 -32a^3 +8a -1) <0 
\end{multline}
for $a> 0$. This shows that $\cL_t^{-1}$ is not Hodge--Riemann as well. 
\end{proof}

It might be interesting to find an example of a non-K\"{a}hler Calabi--Yau manifold with a Hodge--Riemann line bundle. 
\end{eg}	

\begin{eg}\label{eg:Clemensmfd}
For a semistable degeneration $\phi \colon \cY \to \Delta$ of Clemens manifolds as in \cite{MR4085665}, we can check that there is no line bundle $\cL_0$ on the central fiber $\cY_0= \bigcup Y_i$ such that all $\cL_0|_{Y_i}$ are Hodge-Riemann as follows. 

We recall that the central fiber $\cY_0$ is 
\[
\cY_0 = \tilde{X} \cup Q_1 \cup \cdots \cup Q_l, 
\]
where $\tilde{X}$ has $ \nu \colon \tilde{X} \to X$ which is the blow-up along disjoint $(-1,-1)$-curves $C_1, \ldots , C_l \subset X$ on a projective Calabi--Yau 3-fold $X$. 
Let us take $X \subset \bP^4$ is a smooth quintic 3-fold with infinitely many disjoint $(-1,-1)$-curves constructed by Clemens and Friedman (cf. \cite{MR1141199}). 
Let $E_i:= \nu^{-1}(C_i)$ be the exceptional divisor for $i=1,\ldots ,l$.  
We know that $Q_1, \ldots , Q_l$ are isomorphic to a smooth quadric hypersurface $Q \subset \bP^4$. 
Let $\cL_0 \in \Pic Y_0$ and let $\cL_{\tilde{X}}:= \cL_0|_{\tilde{X}}$, $\cL_i:= \cL_0|_{Q_i}$ for $i=1,\ldots, l$. 
Let $d_i:= \cO_X(1) \cdot C_i$ be the degree of the curve $C_i$. 
Then we can write 
\[
\cL_{\tilde{X}} = a \left( \nu^* \cO_X(1) - \sum_{i=1}^l d_i E_i \right), \ \ \cL_i= \cO_{Q_i}(a d_i) 
\]
for some $a \in \bZ$. 
For $a>0$, we can check that $\cL_X$ and $\cL_1, \ldots , \cL_l$ are Lefschetz by elementary calculation. 
However, we can also check that $\cL_{\tilde{X}}$ is not Hodge--Riemann. 
(Indeed, since $b_2 (\cY_t)=0$, we see that the induced line bundle $\cL_t \in \Pic \cY_t$ is trivial and not Lefschetz, thus 
$\cL_{\tilde{X}}$ can not be Hodge--Riemann. )
\end{eg}	

\begin{rem}\label{rem:clemensrho2}
On $\cY_0$ in Example \ref{eg:Clemensmfd}, we can check that the restriction map $\rho_2$ is surjective, 
thus Theorem \ref{thm:intro2} also implies that $H^3$ on Clemens manifold admits a pure Hodge structure. 

On the Hodge--Riemann bilinear relation, we see that 
\[
\Ker \rho_3 / \Image \gamma_3 = H^3(\cY_0^{(1)}, \bC)
\] 
since $\cY_0^{(2)}$ is a union of quadric surfaces and $H^1(\cY_0^{(2)}, \bC) = H^3(\cY_0^{(2)}, \bC)=0$.  
Hence we see that $Q_{W_3}$ is positive definite from the Hodge--Riemann relation on $H^3(\cY_0^{(1)}, \bC)$. 
Moreover, we see that $\Ker \gamma_4=0$ since the restriction map $\rho_2 \colon H^2(\cY_0^{(1)}, \bC) \to H^2(\cY_0^{(2)}, \bC)$ 
is surjective. Hence we can check that $\cY_0$ and $\cY \to \Delta$ satisfy the condition ($\ast$) in Theorem \ref{thm:Hodge-RiemannH^3}. 
\end{rem}

\section*{Acknowledgement}
The author is grateful to Taro Fujisawa, Kenji Hashimoto and Hisashi Kasuya for valuable communications. 
The author would like to thank Kuan-Wen Chen for valuable communications and pointing out mistakes on Proposition \ref{prop:H3CYdesc}(iii) in the previous version. 
He would also like to thank the anonymous referee for pointing out a mistake in the proof of Theorem \ref{thm:Hodge-RiemannH^3} and valuable comments. 
This work was partially supported by JSPS KAKENHI Grant Numbers JP17H06127, JP19K14509, JP23K03032. 


\bibliographystyle{amsalpha}
\bibliography{sanobibs-bddlogcy}

\end{document}